\documentclass[10pt]{amsart}

\usepackage[sc]{mathpazo} 
\usepackage{eulervm}
\usepackage{bbding, amsmath,  amsfonts, amssymb,mathtools}
\usepackage{blindtext} 
\usepackage{bm}
\usepackage{textcomp}
\usepackage{bold-extra}
\usepackage{float}
\usepackage{cancel}
\usepackage{graphicx}
\usepackage{xcolor}
\usepackage[all]{xy}
\usepackage{tikz}
\usepackage{tikz-cd}
\usetikzlibrary{cd}
\usepackage{pgfplots}
\pgfplotsset{width=7cm,compat=1.8}
\usetikzlibrary{arrows,intersections,shapes,snakes,positioning}
\tikzset{
    >=stealth',
    punkt/.style={
           rectangle,
           rounded corners,
           draw=black, very thick,
           text width=8em,
           minimum height=2em,
           text centered},
    pil/.style={
           ->,
           thick}
}
\usepackage{units}
\usepackage{graphicx}
\usepackage{mathrsfs,stmaryrd}
\definecolor{myred}{rgb}{0.5, 0, 0}
\definecolor{linkred}{rgb}{0.8, 0, 0}
\definecolor{mygreen}{rgb}{0, 0.5, 0}
\definecolor{citegreen}{rgb}{0, 0.8, 0}
\definecolor{yqqqqq}{rgb}{0.5019607843137255,0,0}
\definecolor{bluetable}{rgb}{0,0.2,0.5}

\usepackage[T1]{fontenc} 
\usepackage{microtype} 

\usepackage[english]{babel} 

\usepackage[hmarginratio=1:1,top=32mm,columnsep=20pt]{geometry} 
\usepackage[hang, small,labelfont=bf,up,textfont=it,up]{caption} 
\usepackage{booktabs} 
\usepackage{lettrine} 

\usepackage{enumitem} 
\setlist[itemize]{noitemsep} 

\newtheorem{Thm}{Theorem}[subsection]
\newtheorem*{Thm*}{Theorem}
\newtheorem{Lem}{Lemma}[subsection]
\newtheorem{Prop}{Proposition}[subsection]
\newtheorem{Cor}{Corollary}[subsection]
\newtheorem*{Cor*}{Corollary}

\theoremstyle{definition}
\newtheorem{Defn}{Definition}[subsection]
\newtheorem*{Remark*}{Remark}
\newtheorem{Remark}{Remark}[subsection]
\newtheorem*{Notation}{Notation}
\newtheorem{Ex}{Example}
\newtheorem{Exs}[Ex]{Examples}
\numberwithin{equation}{section}

\usepackage[colorlinks=true]{hyperref}
\usepackage{cleveref}

\hypersetup{
    citecolor=citegreen
    ,linkcolor=linkred
    ,urlcolor=bluetable
    }

\AtBeginDocument{%
   \def\MR#1{}
}

\DeclarePairedDelimiter\floor{\lfloor}{\rfloor}

\makeatletter
\newcommand{\oset}[3][0ex]{%
  \mathrel{\mathop{#3}\limits^{
    \vbox to#1{\kern-2\ex@
    \hbox{$\scriptstyle#2$}\vss}}}}
\makeatother
\newcommand{\acfib}{\oset{\sim}{\twoheadrightarrow}}
\newcommand{\we}{\oset{\sim}{\to}}
\newcommand{\accof}{\oset{\sim}{\hookrightarrow}}
\def\mylongmapsto#1{%
\begin{tikzpicture}
\draw (0,0.5mm) -- (0,-0.5mm);
\newlength\mylength
\setlength{\mylength}{\widthof{#1}}
\draw[->] (0,0) -- (1.2\mylength,0) node[above,midway] {#1};
\end{tikzpicture}
}
\newcommand\adjunct[4]{\xymatrix@C=2pc@R=2pc{#1\ar @<1.25ex>[rr]^{#3}&\perp&#2\ar @<1.25ex>[ll]^{#4}}}

\def\A{\mathbb{A}}
\def\Z{\mathbb{Z}}
\def\N{\mathbb{N}}

\def\Part{\mathcal{P}}

\def\C{\mathbb{C}}

\def\Schur{\mathbb{S}}
\def\L{\boldsymbol{L}}
\def\RR{\boldsymbol{R}}
\def\TT{\mathbb{T}}
\newcommand{\PP}{\mathbb P}

\newcommand{\lieg}{\mathfrak{g}}
\newcommand{\gl}{\mathfrak{gl}}
\newcommand{\slin}{\mathfrak{sl}}

\newcommand{\M}{\mathfrak{M}}

\newcommand{\vv}{\mathbf{v}}
\newcommand{\ww}{\mathbf{w}}
\newcommand{\balpha}{{\bm{\alpha}}}
\newcommand{\bbeta}{{\bm{\beta}}}

\newcommand{\Mo}{{\mathfrak{M}^0}}
\def\O{\mathcal{O}}
\newcommand{\TE}{\mathcal{TE}}
\newcommand{\SE}{\mathcal{SE}}
\newcommand{\COF}{{\mathcal{C}\text{\emph{of}}}}

\newcommand{\WE}{\mathcal{WE}}
\newcommand{\FIB}{{\mathcal{F}\text{\emph{ib}}}}
\newcommand{\FIBB}{{\mathcal{F}\text{ib}}}

\newcommand{\llp}[1]{{}^\boxslash\!{#1}}
\def\CC{\mathtt{C}}
\def\DD{\mathtt{D}}
\def\Tcat{\mathtt{T}}

\def\Grp{\mathtt{Grp}}
\def\Vect{\mathtt{Vect}}

\def\dgVect{\mathtt{DGVect}}
\def\Alg{\mathtt{Alg}}
\def\scAlg{\mathtt{sCommAlg}}
\def\Sch{\mathtt{Sch}}
\def\Aff{\mathtt{Aff}}
\def\dAff{\mathtt{d\,Aff}}
\def\DGSch{\mathtt{DGSch}}
\def\DGAff{\mathtt{DGAff}}
\def\Mod{\mathtt{Mod}}
\newcommand{\QCohc}{\mathtt{QCoh}}

\newcommand{\Bimod}{\mathtt{Bimod}}
\def\cAlg{\mathtt{CommAlg}}
\def\Sets{\mathtt{Sets}}
\def\DGA{\mathtt{DGA}}
\def\cDGA{\mathtt{CDGA}}

\def\Ho{{\mathtt{Ho}}}
\def\op{\mathtt{op}}

\newcommand{\dual}{\text{\textasciicaron}}
\newcommand{\fr}{{\mathrm fr}}
\newcommand{\vir}{{\mathrm vir}}

\newcommand{\cyc}{{\mathrm cyc}}
\newcommand{\pt}{{\mathrm pt}}
\newcommand{\cof}{{\mathrm cof}}
\newcommand{\loc}{{\mathrm loc}}

\newcommand{\st}{{\text{-}\mathrm st}}
\newcommand{\acts}{{\, \curvearrowright\,} }

\newcommand{\Mor}{{\mathrm Mor}}

\newcommand{\Hom}{{\mathrm{Hom}}}
\newcommand{\Gra}{{\mathbb{G}\mathrm{r}}}
\newcommand{\GL}{{\mathrm{GL}}}
\newcommand{\SL}{{\mathrm{SL}}}

\newcommand{\End}{\mathrm{End}}
\newcommand{\Aut}{\mathrm{Aut}}
\newcommand{\Span}{\mathrm{Span}}

\newcommand{\ch}{\mathrm{ch}}

\newcommand{\END}{\underline{\mathrm{End}}}

\newcommand{\Id}{{\mathrm{Id}}}
\newcommand{\Rep}{{\mathrm{Rep}}}
\newcommand{\HH}{{\mathrm{H}}}
\newcommand{\KK}{{\mathrm{K}}}
\newcommand{\Rring}{{\mathcal{R}}}
\newcommand{\Qring}{{\mathcal{Q}}}
\newcommand{\DRep}{{\mathrm{DRep}}}
\newcommand{\ev}{\mathrm{ev}}
\newcommand{\n}{{\natural}}
\newcommand{\nn}{{{\natural} {\natural}}}

\newcommand{\Spec}{{\rm{Spec}}}

\newcommand{\Proj}{{\rm{Proj}}}
\newcommand{\Sym}{{\rm{Sym}}}

\newcommand{\id}{{\rm{Id}}}
\newcommand{\rk}{{\rm{rk}}}

\newcommand{\im}{{\rm{Im}}}

\begin{document}


\title{Derived Representation Schemes and Nakajima Quiver Varieties} 
\author{Stefano D'Alesio} 
\address{Stefano D'Alesio, Departement Mathematik,
ETH Z\"urich,
8092 Z\"urich, Switzerland}
\email{stefano.dalesio@math.ethz.ch}
\date{} 

\subjclass[2010]{Primary 14D21, 16G20; Secondary 16E05, 16E45, 19L47}

\begin{abstract}
\noindent
We introduce a derived representation scheme associated with a quiver, which may be thought of as a derived version of a Nakajima variety. We exhibit an explicit model for the derived representation scheme as a Koszul complex and by doing so we show that it has vanishing higher homology if and only if the moment map defining the corresponding Nakajima variety is flat. In this case we prove a comparison theorem relating isotypical components of the representation scheme to equivariant K-theoretic classes of tautological bundles on the Nakajima variety. As a corollary of this result we obtain some integral formulas present in the mathematical and physical literature since a few years, such as the formula for Nekrasov partition function for the moduli space of framed instantons on $S^4$. On the technical side we extend the theory of relative derived representation schemes by introducing derived partial character schemes associated with reductive subgroups of the general linear group and constructing an equivariant version of the derived representation functor for algebras with a rational action of an algebraic torus. 
\end{abstract} 


\setcounter{tocdepth}{1}     
\maketitle
{
{\color{bluetable}
\hypersetup{linkcolor=bluetable}
\tableofcontents
}
}


\section{Introduction}
\label{sec:zero}

Nakajima quiver varieties are certain Poisson varieties constructed from linear representations of a quiver. They were firstly introduced by Nakajima (\cite{Na2},\cite{Na3}) as a geometric tool to study representations of Kac-Moody algebras. They are also interesting from a purely geometric point of view, being a large class of examples of algebraic symplectic manifolds, many of which have been objects of study on their own (for example flag manifolds, framed moduli spaces of torsion free sheaves on $\mathbb{P}^2$, or a Lie algebra version of the character variety of a Riemann surface --- see \cite{BeSc}). More recent studies have also supported the idea that symplectic resolutions, and in particular hyperkähler reductions such as Nakajima quiver varieties, provide a bridge between enumerative geometry, representation theory and integrable systems (\cite{AgFrOk},\cite{NeOk},\cite{Ok3},\cite{Ok1},\cite{Ok2}). 

Quiver varieties are varieties of representations of a quiver: one fixes a vector space on each vertex of the quiver and then consider the linear space of representations obtained by associating to each arrow of the quiver a linear map. Kronheimer and Nakajima (\cite{KrNa}) have first introduced a framed version, which amounts to doubling the set of vertices and drawing a new arrow from each new vertex to its corresponding old one. One of the reasons for considering framed representations is that they appear naturally in the ADHM construction (\cite{AtHiDrMa}) of solutions of self-dual or antiself-dual Yang-Mills equations on $S^4$. They are also interesting from the point of view of representation theory of Lie algebras because dimension vectors of the framed vertices appear as highest weights of the representations (\cite{Na5}). The framing is equivalent to a simpler operation of adding just one vertex with dimension vector $1$, together with as many arrows to each vertex as the framing dimension (as pointed out in \cite{Cr}), however in this paper we consider the framed version of Nakajima quiver varieties. 

The framed quiver is then doubled, which means that each arrow gets doubled by an arrow that goes in the opposite direction: the linear space of representations becomes now a linear cotangent bundle $\smash{M(Q,\vv,\ww): =\TT^*L(Q^\fr,\vv,\ww)}$ (where $\vv,\ww$ are dimension vectors for, respectively, the original and framing vertices). The gauge group is a general linear group on the original vertices $G=G_\vv$ and there is a moment map 
\[
\mu : M(Q,\vv,\ww) \to \lieg^*
\]
in the form of a generalised ADHM equation. Nakajima quiver varieties are defined as Hamiltonian reductions of this action $\smash{G \acts M(Q,\vv,\ww)}$: either affine Hamiltonian reductions, $\smash{\Mo(Q,\vv,\ww)}=\smash{\mu^{-1}(0)\sslash G}$, or quasi-projective $\smash{\M^\chi(Q,\vv,\ww)= \mu^{-1}(0)\sslash_\chi G}$ with the usual tools of geometric invariant theory (\cite{MuFoKi}). For each choice of a (nontrivial) character $\smash{\chi: G\to \C^\times}$ there is a proper Poisson morphism
\begin{equation}
\label{p}
p:\M^\chi(Q,\vv,\ww) \to \Mo(Q,\vv,\ww)\,,
\end{equation}
which is often, but not always, a symplectic resolution of the singularities of $\Mo$.


\subsection{Outline and results}
\label{subsec:outlineandresults}

In this paper we link these varieties with some (derived) representation schemes. The idea of considering representation schemes is certainly not new, in fact it is motivated by the very first algebraic origin of these varieties (see, for example, representation schemes of preprojective algebras in \cite{CrEtGi} and \cite{EtGi}). However the derived version of representation schemes introduces some new invariants in a natural way.

The theory of representation schemes is recalled in detail in \S~\ref{subsec:classicalrep}. To a (unital, associative) algebra $A \in \Alg_k$ one associates $\Rep_V(A)$, the scheme of finite dimensional representations into a fixed vector space $V$. There is a relative version in which the algebra $A$ comes with a fixed structure $\iota: S \to A$ of algebra over another algebra $S$ with a fixed representation $\rho :S \to \End(V)$ and it is natural to define $\Rep_V(A)$ as the scheme of only those finite dimensional representations which are compatible with $\rho$. 

General definitions and results on representation schemes work well over any field $k$ of characteristic zero, but it is necessary to specialise to $k=\C$ in order to relate them to (Nakajima) quiver varieties, which are algebraic varieties over the complex numbers. The (complex) linear space of representations of a quiver $Q$ is a representation scheme of the form $\Rep_V(A)$, where $A=\C Q$ is the path algebra of the quiver. This fact is a consequence of one of the basic results in the theory of representations of quivers: 

\begin{center}
\emph{There is an equivalence of categories between the category of $\C$-linear representations of a quiver $Q$ and the category of left $\C Q$-modules.}
\end{center}

The construction can be easily adapted to include the framing and the doubling of the quiver, and also the operation of taking the fiber of zero through the moment map. In other words it is possible to write the scheme $\mu^{-1}(0)$ as a representation scheme for the path algebra of the framed, doubled quiver, modulo the ideal $\mathcal{I}_\mu$ defined by the moment map:
\[
\mu^{-1}(0) = \Rep_{\C^\vv\oplus \C^\ww} ( A),\qquad A= \C \overline{Q^\fr} \slash \mathcal{I}_\mu\,,
 \]
where $\smash{\C^\vv=\oplus_a \C^{v_a}}$ is the direct sum of the vector spaces placed on the original vertices of the quiver and $\smash{\C^\ww=\oplus_a \C^{w_a}}$ is the one on the framing. We denote this representation scheme also simply by $\Rep_{\vv,\ww}(A)$. The gauge group by which we take the quotient is $G=G_\vv:=\prod_a \GL_{v_a}(\C)  \subset G_\vv\times G_\ww$. This group also arises naturally in the context of representation functors. It is possible to construct an invariant subfunctor by the group $G$ and by doing so we obtain the affine Nakajima variety as the partial character variety  
\[
\Mo(Q,\vv,\ww) = \mu^{-1}(0)\sslash G = \Rep_{\vv,\ww}^G ( A) \, .
\]
Now that we have such a model for this singular scheme we can try to resolve it using the machinery of model categories and in particular the theory of derived representation schemes (\cite{BeFeRa}, \cite{BeKhRa}): we consider the derived scheme
\[
\DRep_{\vv,\ww} (A) \cong \Rep_{\vv, \ww}(A_\cof)\, ,
\]
where $\smash{A_\cof \acfib A}$ is a cofibrant replacement in the category of differential graded algebras. It is (the homotopy class of) a differential graded scheme of the form $\smash{X=(X_0,\O_{X,\bullet})}$, where $\smash{X_0 \cong M(Q,\vv,\ww) }$ is the vector space of linear representations of the framed, doubled quiver, and $\O_{X,\bullet}$ is a sheaf of dg-algebras whose zero homology gives: 
\[
\pi_0(X) = \Spec\big(\HH_0 (\O_{X,\bullet}) \big) =\mu^{-1}(0)\,.
\]
We exhibit an explicit (minimal) resolution $A_\cof \acfib A$ for which this derived representation scheme is a well-known object when it comes to studying resolutions of a singular locus:
\begin{Thm*} [\ref{thm:drep=koszul} in \S~\ref{subsec:zerolocuskoszul}]
There is a cofibrant resolution $A_\cof \acfib A \in \DGA_S$ which gives a model for the derived representation scheme as the (spectrum of the) Koszul complex on the moment map:
\begin{equation}
\label{drep=koszul}
\DRep_{\vv,\ww} \big( A \big) \cong \Rep_{\vv,\ww} \big(  A_\cof) = \Spec \big( \O(M(Q,\vv,\ww))  \otimes \Lambda^\bullet \lieg \big)\, .
\end{equation}
\end{Thm*}

A somewhat natural question is whether or not there is any relationship between Nakajima resolutions \eqref{p} and these derived schemes, and if it is possible to obtain informations about one of the two from the other:
\begin{equation}
\label{diagram}
\begin{tikzpicture}[node distance=2cm, auto,]
  \node[punkt] (Quiver) {Quiver Q};
   \node[punkt, right= of Quiver, right=1.5cm ] (Affine) {Affine Nakajima variety $\Mo = \mu^{-1}(0) 
   \sslash G$ };
   \node[below=of Affine,above=-0.2cm] (dummy1) {};
 \node[above=of Affine,above=-0.1cm] (dummy2) {};
 \node[punkt, right= of dummy1,right=1.9cm] (der) {Derived representation schemes};
  \node[punkt, right= of dummy2,right=1.9cm] (naka) {GIT Nakajima varieties $\M^\chi$};
   \draw (Quiver)   edge[pil,->, decorate, decoration =snake] (Affine);
 \draw (Affine)   edge[pil,->,bend left=45] node[right=1cm,above=-0.3cm]{\parbox{13em}{geometric \\ resolution}} (naka);
  \draw (Affine)   edge[pil,->,bend right=45] node[left=-3.3cm]{\parbox{15em}{algebraic \\ resolution}}   (der);
\node[right=of Affine, right =5cm] (relation) {\color{myred}Relationship?};
\draw (relation) edge [pil,->,color=myred] (naka);
\draw (relation) edge [pil,->,color=myred] (der);
 \end{tikzpicture}
\end{equation}

A first answer is a close relationship (an equivalence) between the condition of flatness for the moment map (which assures that $\M^\chi \to \Mo$ is indeed a resolution, for well-behaved characters $\chi$), and the derived representation schemes to have vanishing higher homologies:
\begin{Thm*} [\ref{thm:0} in \S~\ref{subsec:classicalresultsnaka}]
\label{thm:flatacyclic}
The derived representation scheme $\smash{\DRep_{\vv,\ww}\big(A \big)}$ has vanishing higher homologies if and only if $\mu^{-1}(0)\subset M(Q,\vv,\ww)$ is a complete intersection, which happens if and only if the moment map is flat.
\end{Thm*}

We remark that in general it might not be easy to compute homologies of derived representation schemes, and even just to predict until which degree the homology is nontrivial. Nevertheless, in this special situation it is possible to give a sufficient and necessary condition for the vanishing of higher homologies based on a geometric property (flatness) of the moment map. The importance of Theorem~\ref{thm:flatacyclic} is that there is a combinatorial criterium on the dimension vectors $\vv,\ww$ (proved by Crawley-Boevey, \cite{Cr}, based on the canonical decomposition of Kac, \cite{Kac}) for the flatness of the moment map for representations of quivers.

A second answer to the question in \eqref{diagram} comes when we compare some invariants associated with the derived representation schemes with others associated with the varieties $\M^\chi$. A natural choice is to consider tautological sheaves on the GIT quotient $\M^\chi$ constructed with the usual machinery developed by Kirwan (\S~\ref{subsec:kirwan}). Because of reductiveness of the gauge group $G$ we restrict to consider only tautological sheaves of the form $\mathcal{V}_\lambda$ induced from irreducible representations $V_\lambda$ of $G$. The push-forward of these sheaves in the K-theory of the affine Nakajima variety through the map \eqref{p} computes their ($T$-)equivariant Euler characteristics:
\begin{equation}
\label{1}
\chi_T(\M^\chi,\mathcal{V}_\lambda) \in \KK_T(\Mo)\, ,
\end{equation}
where $T=T_\ww \times T_\hbar$ is the product of the standard maximal torus in the other general linear group on the framing vertices $T_\ww \subset G_\ww$ and a $2$-dimensional torus $T_\hbar$ rescaling the symplectic form and the cotangent direction.

On the other hand also the representation homology $\HH_\bullet(A,\vv,\ww)$ (the homology of the derived representation scheme) is naturally a $G$-module and therefore decomposes into the direct sum of its isotypical components:
\begin{equation}
\HH_\bullet(A,\vv,\ww) = \bigoplus\limits_{\lambda} \Hom_{G}\big(V_\lambda, \HH_\bullet(A,\vv,\ww)\big) \otimes V_{\lambda} \, .
\end{equation}

The isotypical components $\Hom_G(V_\lambda,\HH_\bullet(A,\vv,\ww))$ are modules over the $G$-invariant zeroth homology $\HH_0(A,\vv,\ww)^G=\O(\mu^{-1}(0))^G$ and therefore their Euler characteristics define invariants in
\begin{equation}
\label{2}
\chi_T^\lambda(A,\vv,\ww) = \sum\limits_{i\geq 0 }(-1)^i \big[\Hom_G(V_\lambda, \HH_i(A,\vv,\ww))\big] \in \KK_T(\Mo)\, .
\end{equation}

It is tempting to compare the invariants defined in~\eqref{1} and~\eqref{2}, and the main results of this paper go in this direction. 
First of all, when we consider the trivial representation $V_\lambda=\C$, we prove that if the moment map is flat, then the two invariants are indeed equal:

\begin{Thm*} [\ref{thm:1} in \S~\ref{subsec:comparisontheorem}]
\label{thm:flatstr}
Let $\vv,\ww$ be dimension vectors for which the moment map is flat and let $\chi$ such that $\smash{\M^\chi(Q,\vv,\ww)}$ is a smooth variety (and therefore a resolution of $\smash{\Mo(Q,\vv,\ww)}$). Then we have
\begin{equation}
\label{flatstr}
p_*\big( [\O_{\M^\chi(Q,\vv,\ww)}] \big) = [ \O_{\Mo(Q,\vv,\ww)} ] = \chi_T^G(A,\vv,\ww)  \in \KK_T(\Mo(Q,\vv,\ww))\, .
\end{equation}
\end{Thm*}

When we consider the Hilbert-Poincaré series of~\eqref{flatstr} we obtain an integral formula for the $T$-character of the ring of functions on the GIT quotient $\M^\chi$, that has the following form

\begin{equation}
\label{git=int}
\ch_T(\O(\M^\chi)) = \ch_T(\O(\Mo)) = \frac{1}{|W|} \int_{T_\vv} \frac{\prod_i (1-\hbar_1\hbar_2 r_i)}{\prod_j (1-s_j) } \Delta(x) d x\, ,
\end{equation}

where $r_i=r_i(x)$ and $s_j=s_j(x,t)$ are characters for $T_\vv$ and $T_\vv \times T$, respectively, $\Delta(x)$ is the Weyl factor for $G_\vv$ and the integration is over the compact real form of $T_\vv$ (see \S~\ref{subsec:comparisontheorem} for a more detailed explanation).

Integral formulas of similar flavours have already appeared under different names, both in the mathematical literature (Jeffrey-Kirwan integral/residue formula for GIT quotients --- \cite{JeKi}) and in the physical literature (integral formula for Nekrasov partition function --- \cite{NaYo1},\cite{NaYo2} --- proven, for example, in Appendix A in \cite{FeMu}). We could say that this is not a coincidence, in fact recognising the right-hand side of~\eqref{git=int} in the known example of the Jordan quiver (Nekrasov partition function) as the Euler characteristic of the representation homology was one of the motivations of this project.

For what concerns other tautological sheaves $\mathcal{V}_\lambda$ an equality of the same flavour of~\eqref{flatstr} is true only for large enough $\lambda$, where the definition of largeness depends on the quiver, the dimension vectors $\vv,\ww$ and, perhaps more importantly, also on the GIT parameter $\chi$ (see \S~\ref{subsec:int2}):

\begin{Thm*}[\ref{thm:2} in \S~\ref{subsec:int2}]
Let $\vv,\ww$ be dimension vectors for which the moment map is flat, and $\chi$ a character for which $\M^\chi(Q,\vv,\ww)$ is smooth. For $\lambda$ large enough (Definition~\ref{def:large}) we have
\begin{equation}
\label{flatlamb}
p_*([\mathcal{V}_\lambda]) = [ \HH^0(\M^\chi(Q,\vv,\ww),\mathcal{V}_\lambda)] = \chi_T^{\lambda^*} (A,\vv,\ww) \in \KK_T(\M^0(Q,\vv,\ww))\, .
\end{equation}
\end{Thm*}

Once again by taking the Hilbert-Poincaré series of~\eqref{flatlamb} we obtain a second integral formula for tautological sheaves on the GIT quotient:

\begin{equation}
\label{git=int2}
\ch_T(\chi_T(\M^\chi,\mathcal{V}_\lambda) ) = \ch_T(\HH^0(\M^\chi, \mathcal{V}_\lambda))  = \frac{1}{|W|} \int_{T_\vv} \frac{\prod_i (1-\hbar_1\hbar_2 r_i)}{\prod_j (1-s_j) } f_\lambda(x) \Delta(x) d x\, ,
\end{equation}
where $f_\lambda(x) = \ch_{T_\vv}(V_\lambda)$ is a product of Schur polynomials.

\subsection{Layout of the paper}
\label{subsec:layout}

In \S~\ref{sec:derivedrepresentationschemes} we introduce the general theory of (derived) representation schemes of an algebra. First we recall the theory of representation schemes with some examples, in particular the linear space of representations of a quiver as a representation scheme for its path algebra. Then we recall the derived version introduced by \cite{BeKhRa} and \cite{BeFeRa}. We introduce a more general way to take invariant subfunctors and an equivariant version of derived representation schemes for an action of an algebraic torus which is  useful for our purposes. We decompose the representation homology in isotypical components and define new invariants in the $\KK$-theory of the classical character scheme.

In \S~\ref{sec:dgrepschemesfornakajma} we recall the construction of Nakajima quiver varieties and we show how to view the affine Nakajima variety $\Mo$ as a partial character scheme (a quotient of a representation scheme) for the algebra $\smash{A:= \C \overline{Q^\fr}\slash \mathcal{I}_\mu}$. We construct the derived scheme associated to it and we use the invariants defined in \S~\ref{sec:derivedrepresentationschemes} to decompose the representation homology into classes in the $\KK$-theory of $\Mo$.
In \S~\ref{subsec:cofres} we construct an explicit cofibrant resolution $\smash{A_\cof \acfib A}$ that gives us a concrete model for the derived representation scheme as the (spectrum of the) Koszul complex on the moment map. Therefore we recall some classical properties of the Koszul complex and commutative complete intersections.

In \S~\ref{sec:comparisonresults} we explain the main results of this paper. First we observe that, using the model found in \S~\ref{subsec:cofres}, the derived representation scheme has vanishing higher homologies if and only if the moment map is flat, which is a combinatorical condition on the dimension vectors of the quiver (\cite{Cr}). We recall the definition of tautological sheaves on GIT quotients by the Kirwan map and prove results that compare them with the isotypical components of the representation homology (\eqref{flatstr} and~\eqref{flatlamb}). In particular we obtain some interesting integral formulas (\eqref{git=int} and~\eqref{git=int2}).

In \S~\ref{sec:examples} we show some concrete examples, such as the quiver $A_1$ for which Nakajima varieties are cotangent spaces of Grassmannians, the Jordan quiver for which we obtain framed moduli space of torsion free sheaves on $\PP^2$, and the quiver $A_{n-1}$ with some special dimension vectors for which we obtain the symplectic dual $(\TT^*\PP^{n-1})\dual$, and compute some of the integral formulas that we have proved before.

In Appendix~\ref{app:A} we construct a model structure on equivariant dg-algebras that we need in \S~\ref{subsec:Tenrich}, and in Appendix~\ref{app:B} we recall the theory of irreducible representations for a product of general linear groups as multipartitions, and set some notation that we need in \S~\ref{subsec:int2}.

\begin{Notation} Throughout the paper we denote categories by the standard monospace font: $\Sets$, $\Grp$, $\Vect_k$, $\Alg_k$, \dots The notation used is often both standard and self-explanatory, and when this is not the case we usually recall it in the main body of the paper.
\end{Notation}

\subsection*{Acknowledgements}

I want to express my gratitude to my advisor Giovanni Felder, who introduced me to this subject a couple of years ago and proved numerous times to be a patient, wise and resourceful guide. I also want to mention other people with whom we shared our ideas and contributed with useful comments, in particular Yuri Berest during his brief stay in Zurich, Gabriele Rembado, Matteo Felder and Xiaomeng Xu.



\section{Derived representation schemes of an algebra}
\label{sec:derivedrepresentationschemes}
The family of schemes of finite dimensional representations $\smash{\{\Rep_n(A)\}}_{n\geq 1}$ of an algebra $A$ has been object of study for many years (see for example the early work of Procesi, \cite{Pr}). With the development of noncommutative geometry, they have been seen in a new light when Kontsevich and Rosenberg (\cite{KoRo}) proposed the following principle:
\begin{center}
``\emph{Any noncommutative structure of some kind on A should give an analogous commutative structure on all the representation schemes $\Rep_n(A)$, $n \geq 1$}''.
\end{center} 
This principle seems to work well for (formally) smooth algebras, for which the representation schemes are smooth, but fails in general. The solution proposed in \cite{BeKhRa} is to find a smoothening of representation schemes by extending representation schemes to differential graded algebras, and using the general machinery of model categories to derive them. The purpose of this section is to recall in main details the construction of this derived version of representation schemes from \cite{BeKhRa} and \cite{BeFeRa}, and describe some generalisations that is useful to our purposes.

\subsection{Classical representation schemes}
\label{subsec:classicalrep}

Let $k $ be an algebraically closed field of characteristic zero (later we fix $k = \C$). Let $A \in \Alg_k$ be a unital, associative algebra and $V \in \Vect_k$ a finite dimensional vector space. We consider the functor on unital commutative algebras:

\begin{equation}
\label{repfunctor}
\begin{aligned}
\Rep_V(A): \,\,&\cAlg_k (=\Aff_k^\op) \to \Sets\\
&B \longmapsto \Hom_{\Alg_k}\big(A,\End(V) \otimes_k B\big)\, .
\end{aligned}
\end{equation}
This functor is (co)-representable, by the commutative algebra $\smash{A_V:=\big( \sqrt[V]{A} \big)_{\nn}}$. The two functors $\smash{\sqrt[V]{-}}$ and $\smash{(-)_{\nn}}$ are, respectively, the matrix reduction functor and the abelianisation functor, which are left adjoints to the followings:
\begin{equation}
\label{adjunctions}
\adjunct{\Alg_k}{\Alg_k}{\sqrt[V]{-}}{\End(V)\otimes_k(-)}\, ,
\qquad \quad
\adjunct{\Alg_k}{\cAlg_k}{(-)_\nn}{U}\, .
\end{equation}
Explicit formulas for them are $\smash{\sqrt[V]{A} = \big( \End(V) \ast_k A  \big)^{\End(V)}}$ and $\smash{(C)_{\nn} =C\slash \langle[C,C]\rangle}$, where $\smash{\langle [C,C]\rangle}$ is the $2$-sided ideal generated by the commutators. By combining the two adjunctions in \eqref{adjunctions} we get an adjunction for the representation functor:
\begin{equation}
\adjunct{\Alg_k}{\cAlg_k}{(-)_V}{\End(V)\otimes_k(-)} \, ,
\end{equation}
so that the commutative algebra $A_V$ is uniquely defined by the natural isomorphisms:
\begin{equation}
\label{natiso}
\Hom_{\cAlg_k}(A_V, B) \cong \Hom_{\Alg_k}(A,\End(V)\otimes_k B), \quad \forall B \in \cAlg_k\, .
\end{equation}
\begin{Defn} The affine scheme associated to $A_V \in \cAlg_k$ is the \textit{representation scheme} $\Rep_V(A)= \Spec(A_V) \in \Aff_k$ (strictly speaking we identify it with its functor of points as we originally defined it $\Rep_V(A) \in \O(\Aff_k^\op,\Sets)$ in \eqref{repfunctor}). We recover $A_V = \O(\Rep_V(A))$ as the algebra of functions on the representation scheme.
\end{Defn}

We can assume that $V=k^n$ and write simply $\Rep_n(A)=\Spec(A_n)$ instead of $\Rep_V(A)=\Spec(A_V)$. Let us show some examples:

\begin{Exs}
\label{Exs1}
\begin{enumerate}
\item [(0)] If $A \in \cAlg_k \subset \Alg_k$ is a commutative algebra then clearly from \eqref{natiso}:
\[
A_1 = A  \quad \leftrightarrow \quad \Rep_1(A)=\Spec(A)\, .
\] 
\item The free algebra in $m$ generators $A=F_m=k\langle x_1,\dots ,x_m \rangle$ has no relations and therefore $\Rep_n(F_m)$ is the scheme of $m$-tuples of $n\times n$ matrices:
\[
 \Rep_n(F_m)=M_{n\times n}(k)^{ m}\, .
\]
\item The polynomial algebra $A=k[x_1,\dots, x_m]$ can be expressed as the free algebra in $m$ generators modulo the ideal generated by all commutators $[x_i,x_j]$ and therefore its representation scheme is the closed subscheme of $m$-tuples of $n\times n$ matrices that pairwise commute:
\[
 \Rep_n(A) = C(m,n) :=  \big\{(X_1,\dots, X_m) \in M_{n\times n}(k)^{ m} \,| \, [X_i,X_j] =0 \, \forall i,j \big\}  \subset \Rep_n(F_m)\, .
\]
\item The algebra of dual numbers $A=k[x]/(x^2)$ gives the scheme of square-zero matrices:
\[
\Rep_n(A)= \big\{ X \in M_{n\times n}(k) \, | \, X^2 =0 \big\}\, .
\]
\item The algebra of differential operators on the affine line $A=\mathrm{Diff}(\mathbb{A}_k^1) =k\langle x,d\rangle\slash ([d,x] =1)$ has no representation because if $X,D \in M_{n\times n}(k)$ are matrices satisfying $[D,X]=\mathbb{1}_n$, then taking traces we would get $0=n$, which is absurd:
\[
\Rep_n\big(\mathrm{Diff}(\mathbb{A}_k^1)\big) = \emptyset\, .
\]
\item The algebra of Laurent polynomials in $m$ variables $A=k [t_1^{\pm 1} , \dots, t_m^{\pm 1}]$ is similar to the example of commuting matrices, except that now the matrices are required to be invertible:
\[
\Rep_n(A) =\big\{(X_1,\dots, X_m) \in \GL_n(k)^{ m} \,| \, [X_i,X_j] =0 \, \forall i,j \big\} \, .
\]
\item 
\label{ex5} 
More generally writing any finitely generated algebra as a free algebra modulo some relations 
\[
A = F_m \slash \langle  r_1, \dots , r_s\rangle \,,\qquad r_1,\dots ,r_s \in F_m =k\langle x_1, \dots ,x_m \rangle\, ,
\]
then its representation scheme is identified with the closed subscheme 
\[
\Rep_n(A) = \big\{(X_1,\dots, X_m) \in M_{n\times n}(k)^m \, | \, r_i(X_1,\dots,X_m) = 0 \, \forall i  \big\} \subset \Rep_n(F_m)
\]
of $m$-tuples of $n\times n$ matrices defined by the equations $r_1,\dots ,r_s$.
\end{enumerate}
\end{Exs}
Another fundamental example is that of path algebras of (finite) quivers. These algebras come with an additional structure of algebras over the finite dimensional algebras of their empty paths, which is crucial when considering their representations, therefore we need to consider a \emph{relative} version of representation schemes. Formally we fix an algebra $S\in \Alg_k$ and we consider the under category $S \downarrow\Alg_k$ (also denoted by $\Alg_S$ following the notation of \cite{BeFeRa} and \cite{BeKhRa}) which is the category of algebras $A \in \Alg_k$ together with a fixed morphism $S \to A$. We also fix a representation $\rho:S \to \End(V)$. 

\begin{Notation} Sometimes when we want to remark that $A$ comes with a map from $S$ we denote this object as $S\backslash A \in \Alg_S$. However, when there is no risk of confusion, we just use $A$.
\end{Notation}

With these ingredients it is natural to consider only those representations $A \to \End(V)$ that agree with $\rho$ on $S$. In terms of functor of points this corresponds to 
\begin{equation}
\label{relrep}
\begin{aligned}
\Rep_V(A): \,\,&\cAlg_k \to \Sets\\
&B \longmapsto \Hom_{\Alg_S}\big(A,\End(V) \otimes_k B\big)\, .
\end{aligned}
\end{equation}
This functor is also (co)representable, by the commutative algebra $A_V$ defined as before except for $\ast_k$ substituted by $\ast_S$, the coproduct in $\Alg_S$. Letting $A$ vary we obtain a relative version of the representation functor $(-)_V$, and a similar adjunction
\begin{equation}
\label{reladj}
\adjunct{\Alg_S}{\cAlg_k}{(-)_V}{\End(V)\otimes_k(-)}\, .
\end{equation}
\begin{Ex}[Path algebra of a quiver] 
\label{ex:pathalg}
Let $Q$ be a finite quiver and $A =\C Q \in \Alg_\C$ its path algebra over the complex numbers. What follows works well for any field $k$ of characteristic zero but later we are interested only in $k=\C$. We recall that the path algebra is the free vector space on the admissible paths in the quiver, with product given by concatenation of paths. It has a set of orthogonal idempotents $\{e_i \}_{i \in Q_0} \subset A$:
\[
e_i  e_j = \delta_{ij} e_j\, ,
\]
 which are the empty paths on the vertices, and their sum is the unit of the algebra: $\smash{\sum_{i\in Q_0} e_i =1 \in A}$. We can then consider the subalgebra generated by these idempotents 
 \[
 S=\langle e_i \rangle_{i\in Q_0} = \Span_{\C}\{ e_i\}_{i\in Q_0}\, ,
 \]
with the natural inclusion $\iota :S \to A$. We now fix a dimension vector $\vv \in \N^{Q_0}$ and we consider the linear space of representations of the quiver $Q$ with the complex vector space $\C^{v_i}$ placed at the vertex $i \in Q_0$:
\begin{equation}
\label{linspace}
L(Q,\vv) := \bigoplus_{\gamma \in Q_1} \Hom_\C\big(\C^{v_{s(\gamma)}}, \C^{v_{t(\gamma)}}\big)\, .
\end{equation} 
where $s,t:Q_1 \to Q_0$ are the source and target maps of the quiver. From the algebraic point of view we fix the following representation of $S$ in the vector space $\C^\vv:= \oplus_i \C^{v_i}$:
\begin{equation}
\label{rhov}
\begin{aligned}
\rho=\rho_\vv : \,&S \to \oplus_i \End_\C(\C^{v_i}) \subset \End_\C(\C^\vv)\\
&e_i \longmapsto E_i := 0 \oplus \dots \oplus \underbrace{\mathbb{1}_{\C^{v_i}}}_{\text{i-th factor}}\oplus  \dots \oplus 0\, .
\end{aligned}
\end{equation}
\begin{Prop}
\label{connection2}
The linear space of representations of the quiver $Q$ with fixed dimension vector $\vv$ is isomorphic to the (relative) representation scheme of its path algebra:
\begin{equation}
\label{lqv=rep}
L(Q,\vv) \cong \Rep_{\C^\vv}( \C Q)\, .
\end{equation}
\end{Prop}
\begin{proof} 
Let us consider the complex vector space with basis given by the set of arrows of the quiver $M:=\Span_\C\{ x_\gamma\}_{\gamma \in Q_1}$. It has the structure of an $S$-bimodule, and its tensor algebra is the path algebra of the quiver:
\[
A=\C Q = T_S M := S \oplus M \oplus (M\otimes_S M) \oplus \dots \, .
\]
For a dimension vector $\vv \in \N^{Q_0}$ we consider the graded vector space $\C^\vv = \oplus_i \C^{v_i}$, whose endomorphism algebra $\End_\C(\C^\vv)$ is an $S$-bimodule via the map \eqref{rhov}. By the universal property of the tensor algebra, giving a representation $T_SM \to \End_\C(\C^\vv)$ that agrees with $\rho$ on $S$, is equivalent to give a $S$-bimodule map $M \to \End_\C(\C^\vv)$:
\[
\Hom_{\Alg_S}(A,\End_\C(\C^\vv)) \cong \Hom_{S-\Bimod}(M,\End_\C(\C^\vv) \cong \bigoplus_{\gamma \in Q_1} \Hom_\C \big( \C^{v_{s(\gamma)}} , \C^{v_{t(\gamma)}} \big)  = L(Q,\vv)\, .
\]
\end{proof}
\end{Ex}


\subsection{Derived representation schemes}
\label{subsec:derivedrep}

As already anticipated in the introduction of this Section, the noncommutative geometry principle of transferring a geometric property on an algebra $A$ (e.g. noncommutative complete intersection, Cohen-Macaulay, etc.) on the corresponding commutative one on $\Rep_V(A)$ might fail when $A$ is not a (formally) smooth algebra. This seems to be related to the fact that the functor $\Rep_V(-)$ is not exact.

We discuss the following derived version of representation schemes firstly introduced in \cite{BeKhRa}. The idea is to ``resolve'' the singularities of the representation schemes by using the tools of homological algebra, in the sense of Quillen's derived functors on model categories.

We enlarge the category of algebras to the one of differential graded algebras $\DGA_k$ (in our conventions differentials have always degree $-1$), and as before we consider the under category $\DGA_S:= S\downarrow \DGA_k$ of dg-algebras $A$ with a fixed morphism $S\to A$.

We also fix a differential graded vector space $V \in \dgVect_k$ of finite total dimension, and denote by $\END(V) \in \DGA_k$ the differential graded algebra of endomorphisms, with differential
\begin{equation}
d f = d_V \circ f - (-1)^i f \circ d_V, \quad f \in \END(V)_i\, .
\end{equation}
Moreover we need to fix a representation of $S$ in $V$, that is a dga morphism $\rho : S \to \END(V)$, which makes $\END(V)$ an object of $\DGA_S$. With these ingredients we can define a differential graded version of the representation functor for $A \in \DGA_S$ as the functor from commutative dg-algebras:
\begin{equation}
\label{dgrepfunctor}
\begin{aligned}
\Rep_V(A): \,\,&\cDGA_k \to \Sets \\
&B \longmapsto \Hom_{\DGA_S}(A,\END(V) \otimes_k B)\, .
\end{aligned}
\end{equation}
\begin{Remark}
We use the same notation as in the non-graded case because in the particular case of $S,A,V$ being concentrated in degree zero we recover the same functor as before (when restricted to $\Alg_k \subset \DGA_k$).
\end{Remark}
This functor is also (co)-representable, by the object $\smash{A_V:= \big( \sqrt[V]{A} \big)_\nn}$ constructed in the same way as before, with 
\begin{equation}
\sqrt[V]{ A } = \left( \END(V) \ast_S A \right)^{\END(V)} \, ,
\end{equation}
where $\ast_S$ is the free product over $S$, the categorical coproduct in $\DGA_S$. As before we obtain a pair of adjoint functors
\begin{equation}
\label{quillenpair}
\adjunct{\DGA_S}{\cDGA_k}{(-)_V}{\END(V)\otimes_k (-)}\, .
\end{equation}
These categories have model structures for which this adjunction is a Quillen adjunction, and therefore produces a total right-derived functor $\RR \big(\END(V)\otimes_k(-)\big)$, but more importantly a left-derived functor $\L(-)_V$ that we use to define the derived representation scheme.

We consider on $\DGA_k$ and $\cDGA_k$ the so-called projective model structures for which weak equivalences are quasi-isomorphisms of complexes and fibrations are degree-wise surjective maps (Theorem 4 in \cite{BeFeRa}). It is useful for later purposes to consider also the categories $\DGA_k^+$ and $\cDGA_k^+$, which are the categories of non-negatively graded differential graded and commutative differential graded algebras, respectively, and with their projective model structures with the only difference that now fibrations are degree-wise surjective maps in all (strictly) \emph{positive} degrees. All these categories are fibrant (every object is fibrant), with initial object $k$ and final object $0$.

The category $\DGA_S$ is an example of an under category (category in which objects are objects of the original category coming with a fixed morphism from the object $S$ in this case). As such it comes with a forgetful functor $\DGA_S \to \DGA_k$ and the model structure on $\DGA_S$ is the one in which weak-equivalences, fibrations and cofibrations are exactly the maps which are sent to weak-equivalences, fibrations and cofibrations via the forgetful functor. Clearly also the under category $\DGA_S$ is fibrant, with final object still $0$ (viewed as an object of $\DGA_S$ via the unique map $S\to 0$), and initial object $S$ (viewed as an object of $\DGA_S$ via the identity map $\id_S :S \to S$).

For a model category $\CC$, we denote by $\Ho(\CC)$ its homotopy category and by $\gamma : \CC \to \Ho(\CC)$ the canonical functor.

\begin{Thm}[Theorem 7 in \cite{BeFeRa}] 
\label{thm:L(A)_V}
\begin{itemize}
\item [(i)] The pair of functors in \eqref{quillenpair} form a Quillen pair.
\item [(ii)] The representation functor $(-)_V$ has a total left derived functor given by
\begin{equation}
\label{leftder}
\begin{aligned}
\L (-)_V : \,\,&\Ho(\DGA_S) \to \Ho(\cDGA_k) \\
& \begin{cases} A \longmapsto  \big( A_\cof \big)_V\\
\gamma f \longmapsto \gamma (\tilde{f})_V
\end{cases} 
\end{aligned}
\end{equation}
where $\smash{A_\cof \acfib A}$ is a cofibrant replacement in $\DGA_S$, and for a morphism $f:A \to B$, the morphism $\smash{\tilde{f}: A_\cof \to B_\cof }$ is a lifting of $f$ between the cofibrant replacements.
\item [(iii)] For any $A \in \DGA_S$ and any $B \in \cDGA_k$ there is a canonical isomorphism:
\begin{equation} 
\label{caniso}
\Hom_{\Ho(\cDGA_k)}( \L (A)_V, B) \cong \Hom_{\Ho(\DGA_S)} (A, \END(V)\otimes_k B) \, .
\end{equation}
\end{itemize}
\end{Thm}
\begin{Defn} 
\label{def1}
For $S \in \Alg_k$ concentrated in degree $0$, the following composite functor
\begin{equation}
\label{drep}
\begin{aligned}
&\Alg_S  \to \Ho(\DGA_S) \xrightarrow[]{\L(-)_V}    \Ho(\cDGA_k)  \\
&\,\,\, A\,\, \mylongmapsto{\qquad\qquad \quad \qquad}\,\,\,\,\, \L(A)_V
\end{aligned}
\end{equation}
is called \emph{derived representation functor}. The homology of the (homotopy class of the) commutative differential graded algebra $\L(A)_V \in \Ho(\cDGA_k)$ depends only on $A\in \Alg_S$ and $V$. It is called the \textit{representation homology} of $ A$ with coefficients in $V$:
\begin{equation}
\HH_\bullet(A,V) := \HH_\bullet(\L(A)_V) \,.
 \end{equation}
\end{Defn}
\begin{Remark}
\label{rmk:zerohom}
By its definition, the zero-th homology recovers the classical representation scheme (see Theorem 9 in \cite{BeFeRa}):
\begin{equation}
 \HH_0 \big( A,V ) \cong A_V= \O\big( \Rep_V(A)\big)\, .
\end{equation}
\end{Remark}
As we anticipated before, we are interested in a slightly different version of this story: if we start from a vector space $V$ concentrated in degree $0$ and $S\in \Alg_k$ then the previous pair \eqref{quillenpair} restricts to a pair of functors
\begin{equation}
\label{quillenpair0}
\adjunct{\DGA_S^+}{\cDGA_k^+}{(-)_V}{\END(V)\otimes_k (-)}\, ,
\end{equation}
which is still a Quillen pair, and the analogous result of Theorem~\ref{thm:L(A)_V} holds. We give a second definition of:
\begin{Defn}
\label{def2}
The \emph{derived representation functor} is the following functor:
\begin{equation}
\label{drep0}
\Alg_S \to \Ho(\DGA_S^+) \xrightarrow[]{\L(-)_V} \Ho(\cDGA_k^+)\, .
\end{equation}
The \emph{representation homology} of the relative algebra $A \in \Alg_S$ is the homology of $\L(A)_V \in \Ho(\cDGA_k^+)$.
\end{Defn}

\begin{Remark} 
\label{rem:+}
Definition~\ref{def1} and \eqref{def2} are not really different. In fact, there is an adjunction between the categories $\DGA_S^+$ and $\DGA_S$ 
\begin{equation}
\label{iotatrunc}
\adjunct{\DGA_S^+ }{\DGA_S}{\iota}{\tau} \, ,
\end{equation}
where the functor $\iota$ is the obvious inclusion and the functor $\tau$ is the one that sends an unbounded differential graded algebra $A \in \DGA_S$ to its truncation:
\[
\tau(A) =  [ \cdots \to A_2 \xrightarrow{d_2} A_1 \xrightarrow{d_1} \ker(d_0)  ] \in \DGA_S^+\, .
\]
It is straightforward to see that $\tau$ preserves fibrations and weak equivalences, and dually the map $\iota$ preserves cofibrations and weak equivalences, in particular it sends cofibrant objects to cofibrant objects. Now let $A \in \Alg_S$ and choose a cofibrant replacement $\smash{Q\acfib A \in \DGA_S^+}$. A priori this map is only surjective in positive degrees, but because $A$ is concentrated in degree $0$, we have $A=\HH_0(A)$, and the isomorphism in homology $\smash{\HH_0(Q) \cong \HH_0(A)}$ proves that it is surjective also in degree $0$, so still a fibration in $\DGA_S$. In other words the cofibrant replacement $\smash{Q\acfib A}$ is still a cofibrant replacement in $\DGA_S$ and therefore it can be used to compute the derived representation functor \eqref{drep}, showing that Definition \ref{def1} is equivalent to Definition \ref{def2}.
\end{Remark}
\begin{Remark} [\textbf{The dual language of dg-schemes}] Another reason for considering the category $\cDGA_k^+$ instead of $\cDGA_k$ is that it is anti-equivalent to the category of differential graded schemes, as introduced by I. Ciocan-Fontanine and M. Kapranov in \cite{CiKa2}. We recall their definition of dg-schemes (over $k$) as a pair $X=(X_0, \O_{X,\bullet})$, where $X_0$ is an ordinary scheme over $k$ and $\O_{X,\bullet}$ is a sheaf of non-negatively graded commutative dg-algebras on $X_0$ such that the degree zero is $\O_{X,0}=\O_{X_0}$ the structure sheaf of the classical scheme $X_0$ and each $\O_{X,i}$ is quasicoherent over $\O_{X,0}$. A morphism of dg-schemes over $k$ is just a morphism of dg-ringed spaces $f:X=(X_0,\O_{X,\bullet}) \to Y=(Y_0,\O_{Y,\bullet})$, and this makes $\DGSch_k$ into a category. A dg-scheme $X$ is called \emph{affine} if the underlying classical scheme $X_0$ is affine. The full subcategory of dg-affine schemes $\DGAff_k \subset \DGSch_k$ is antiequivalent to the category $\cDGA_k^+$, via the the equivalence of categories:
\begin{equation}
\label{dgaff}
\adjunct{\DGAff_k^\op}{\cDGA_k^+}{\Gamma( - ) }{\Spec}\, ,
\end{equation}
where $\Gamma(-)$ is the functor taking a dg-affine $X$ into the global sections of the sheaf $\O_{X,\bullet}$ (degreewise), and $\Spec$ is the dg-spectrum sending a commutative dg-algebra $A$ to the classical scheme $X_0=\Spec(A_0)$ together with the quasicoherent sheaves $\O_{X,i}$ associated to the modules $A_i$ via the correspondence $\QCohc_{X_0} \cong \Mod_{A_0}$. These names are motivated by the fact that the previous equivalence restricts to the classical equivalence of categories
\begin{equation}
\label{aff}
\adjunct{\Aff_k^\op}{\cAlg_k}{\Gamma( - ) }{\Spec}\, .
\end{equation}
This definition of dg-affine schemes coincides with Toën-Vezzosi's definition of derived schemes $\dAff_k^\op = \scAlg_k$ as simplicial commutative algebras (\cite{To}) because over a field $k$ of characteristic zero they are equivalent to commutative dg-algebras.

The equivalence of categories \eqref{dgaff} can be trivially used to transfer the projective model structure on commutative dg-algebras to the category of dg-affine schemes. Obviously the pair $(\Gamma(-),\Spec)$ becomes a Quillen equivalence, i.e. an equivalence on the homotopy categories:
\begin{equation}
\label{hodgaff}
\adjunct{\Ho(\DGAff_k^\op)}{\Ho(\cDGA_k^+)}{\L\Gamma( - ) }{\RR\Spec}\, .
\end{equation}
Moreover because every object in $\cDGA_k^+$ is fibrant, the derived spectrum $\RR\Spec$ actually coincides with the underived $\Spec$ on the objects. 
\end{Remark}
\begin{Defn} The \emph{derived representation scheme} of the relative algebra $A \in \Alg_S$ in a vector space $V$ is the object $\DRep_V(A) \in \Ho(\DGAff_k)$ obtained applying to $A$ the following composition of functors:
\begin{equation}
\DRep_V(-):\Alg_S \to \Ho(\DGA_S^+) \xrightarrow{\L(-)_V} \Ho(\cDGA_k^+) \xrightarrow{\RR\Spec} \Ho(\DGAff_k) \, .
\end{equation}
\end{Defn}
This definition differs from the one given in \cite{BeKhRa} and \cite{BeFeRa} only from the last composition with the derived spectrum functor. The reason we have to do so is to be consistent with the notation for the classical representation scheme $\Rep_V(A) \in \Aff_k$.
\begin{Remark} Because every object in $\cDGA_k^+$ is fibrant, the derived representation scheme $\smash{\DRep_V(A)}$ is simply
\begin{equation}
\label{eq:drep}
\DRep_V(A) = \RR\Spec(\L(A)_V)= \Spec(\L(A)_V)=\Spec((A_\cof)_V)= \Rep_V(A_\cof)\, ,
\end{equation}
where $\smash{A_\cof \acfib A \in \DGA_S^+}$ is a cofibrant replacement. Different choices of cofibrant replacements give different models to $\DRep_V(A)$, which are weakly equivalent to each other. In what follows we choose one specific model for $\DRep_V(A)$ obtained through a choice of a preferred cofibrant replacement. Strictly speaking in \eqref{eq:drep} we should write $\smash{\DRep_V(A) = \gamma \Rep_V(A_\cof)} \in \Ho(\DGAff_k)$ to remember that we are considering the homotopy class, but we make an abuse of notation by dropping $\gamma$.
\end{Remark}

\begin{Exs}
\label{Exs2}
In the following examples we describe explicit cofibrant resolutions for some of the algebras in the Examples~\ref{Exs1} and give a model for their derived representation schemes with value in a vector space $V$ concentrated in degree $0$ (therefore we still use the notation $\DRep_n(-)=\DRep_V(-)$ for $V=k^n$).
\begin{enumerate}
\item The free algebra in $m$ generators $A=F_m$ is already a cofibrant object in $\DGA_k^+$ because it is free, therefore
\[
\DRep_n(F_m) \cong \Rep_n(F_m) \cong M_{n\times n}(k)^{ m}\, .
\]
\item The commutative algebra in two variables $A=k[x,y]$ is not cofibrant because of the relation $[x,y]=0$. It turns out that it suffices to add one variable $\vartheta$ in homological degree $1$ that kills this relation ($d \vartheta = [x,y]$) to obtain a cofibrant replacement:
\[
A_\cof := k\langle x,y,\vartheta \rangle  \acfib A = k[x,y] \, ,
\]
and therefore the derived representation scheme is the nothing else but the (spectrum of the) Koszul complex for the scheme of $n\times n$ commuting matrices:
\[
\DRep_n(A) \cong \Rep_n(A_\cof) =\Spec \big( k[ x_{ij}, y_{ij},\vartheta_{ij}]_{i,j=1}^n 
\big) \, ,\qquad d\vartheta_{ij} = \sum_k x_{ik}y_{kj} -y_{ik}x_{kj}\, .
\]
\item Calabi--Yau algebras of dimension $3$ (see \cite[\S~1.3]{Gi3}). Consider the free algebra $F_m$ and the commutator quotient space of cyclic words: $(F_m)_\cyc =F_m \slash [F_m,F_m]$. M. Kontsevich introduced linear maps $\partial_i : (F_m)_\cyc \to F_m$ for each $i=1,\dots, m$ which we can use, together with a potential $\Phi \in (F_m)_\cyc$, to define the algebra
\begin{equation}
\label{algpot}
A= \mathfrak{U}(F_m,\Phi) := F_m \slash ( \partial_i \Phi )_{i=1,\dots,m}\, ,
\end{equation}
which is the quotient of the free algebra $F_m$ by the two-sided ideal generated by the partial derivatives of the potential $\Phi$. For example when $m=3$, $F_3=k \langle x,y,z \rangle$ and observe that the partial derivatives for the potential $\Phi = xyz-yxz$ give the commutators, therefore $A= k[x,y,z]$ is the polynomial ring in $3$ variables. For an algebra defined by a potential as above in \eqref{algpot} we define the following dg-algebra:
\begin{equation}
\label{dalgpot}
\begin{aligned}
& \mathfrak{D}(F_m,\Phi) := k \langle x_1,\dots, x_m, \vartheta_1,\dots, \vartheta_m, t \rangle\, ,\\
&(\mathrm{deg}(x_i,\vartheta_i,t) = (0,1,2))\quad d \vartheta_i = \partial_i \Phi, \quad d t = \sum\limits_{i=1}^m [x_i,\vartheta_i]\, .
\end{aligned}
\end{equation}
Ginzburg explains in \cite{Gi3} how Calabi--Yau algebras of dimension $3$ are all of the form \eqref{algpot} and they are exactly those for which a suitable completion of $\mathfrak{D}(F,\Phi)$ is a cofibrant resolution. This is in particular true for the example of polynomials in $3$ variables (see example 6.3.2. in \cite{BeFeRa}), for which no completion is needed and:
\[
\DRep_n ( k[x,y,z]) \cong \Rep_n( k[x,y,z,\xi,\vartheta,\lambda, t]) = \Spec\big( k[x_{ij},y_{ij},z_{ij},\xi_{ij},\vartheta_{ij},\lambda_{ij}, t_{ij}]_{i,j=1}^n\big)\, ,
\]
where the variables $\xi,\vartheta,\lambda$ are the ones we called $\vartheta_1,\vartheta_2,\vartheta_3$ in \eqref{dalgpot}.
\end{enumerate}
\end{Exs}


\subsection{G-invariants and isotypical components}

\label{subsec:ginvariants}

On the (derived) representation scheme there is a natural action of the general linear group $\GL(V)$ by which one can consider the associated character scheme of invariants. Later we consider only invariants by a subgroup $G\subset \GL(V)$, therefore we propose the following theory of partial invariant subfunctors by $G$ that generalises the theory introduced in \cite[\S~2.3.5]{BeKhRa} and in \cite[\S~3.4]{BeFeRa} in the absolute case $S=k$. However we point out that the results of this section are strongly inspired by \cite{BeKhRa} and \cite{BeFeRa}, which already contain most of the material needed.

Suppose that both $V$ and $S$ are concentrated in degree 0, $\rho : S\to \End(V)$ is a fixed representation and consider
\[
G_S := \{ g \in \GL(V) \, | \, g^{-1} \rho(s) g = \rho (s) \,\,\, \forall s \in S \}\,,
\]
the subgroup of $\rho$-preserving transformations. Observe that in the absolute case $ S=k$ then $G_S= \GL(V)$. Now consider any reductive subgroup $G \subset G_S $, whose right action on $\End(V)$ extends to the functor:
\[
\End(V) \otimes_k (-) : \cDGA_k \to \DGA_S 
\]
(for this we need that $G$ consists of transformations which all preserve $\rho$). And consequently we obtain a left action on $(-)_V : \ \DGA_S \to \cDGA_k$ and we can consider the invariant subfunctor 
\begin{equation}
\label{invsubf}
\begin{aligned}
(-)_V^G : \,&\DGA_S \to \cDGA_k\\
&A \longmapsto A_V^G\, .
\end{aligned}
\end{equation}
As explained in \cite{BeKhRa}, unlike $(-)_V$, the functor $(-)_V^G$ does not seem to have a right adjoint, so we cannot prove that it has a left derived functor from Quillen's adjunction theorem. Nevertheless we can prove that such a left derived functor exists:
\begin{Thm}
\label{thm:derivedV}
\begin{itemize}
\item [(a)] $(-)_V^G : \DGA_S \to \cDGA_k$ has a total left derived functor $\L(-)_V^G$.
\item [(b)] For every $A \in \DGA_S$ there is a natural isomorphism:
\begin{equation}
\HH_\bullet [\L(A)_V^G]  \cong \HH_\bullet (A,V)^G\, .
\end{equation}
\end{itemize}
\end{Thm}

To prove this theorem it is convenient to recall a few notions/results. Let $\Omega = k[t]\oplus k[t] dt$ be the algebraic de Rham complex of the affine line $\mathbb{A}^1_{k}$ (in our conventions differentials have degree $-1$ and therefore $dt$ has the wrong degree $-1$). We define a \emph{polynomial homotopy} between $f,g:A \to B \in \DGA_S$ as a morphism $h:A \to B \otimes \Omega \in \DGA_S$, such that $h(0)=f$ and $h(1)=g$, where for each $a\in k$, $h(a)$ is the following composite map:
\[
h(a) : A \xrightarrow[]{h} B \otimes \Omega \xrightarrow[]{\pi} B \otimes \Omega\slash(t-a) \cong B\otimes k = B\, .
\]
The reason why polynomial homotopy is equivalent to the homotopy equivalence relation in $\DGA_S$ is explained in Proposition B.2. in \cite{BeKhRa}.
\begin{Lem}
\label{lem:polhom}
Let $h:A \to B\otimes \Omega \in \DGA_S$ be a polynomial homotopy between $f,g:A \to B$. Then:
\begin{enumerate}
\item There is a homotopy $h_V :A_V \to B_V\otimes \Omega \in \cDGA_k$ between $h_V(0) = f_V$ and $h_V(1)=g_V$.
\item $h_V$ restricts to a morphism $h_V^{G}: A_V^G \to B_V^G \otimes \Omega \in \cDGA_k$.
\end{enumerate}
\end{Lem}
\begin{Remark} It is important to observe that, despite the misleading notation, the map $h_V$ in part (1) is \emph{not} the map obtained applying the functor $(-)_V$ to the map $h$. The latter would in fact be a map $A_V \to (B\otimes\Omega)_V \neq B_V\otimes \Omega$. The same thing applies for the map $h_V^G$ in part (2), which is \emph{not} the map obtained applying the functor $(-)_V^G$ to the map $h$.
\end{Remark}
\begin{proof}
We omit the proof because it is analogous to the proof of Lemma 2.5 in \cite{BeKhRa}. 
\end{proof}

\begin{proof} [Proof of Theorem~\ref{thm:derivedV}] \begin{itemize}
\item [(a)] By Brown's lemma (Lemma A.2 in \cite{BeKhRa}) it is sufficient to prove that the functor $(-)_V^G$ maps acyclic cofibrations between cofibrant objects to weak equivalences. Let $\smash{i :A \accof B}$ such an acyclic cofibration between cofibrant objects $A,B\in \DGA_S$. Then there is a map $p:B \to A$ such that $p \circ i = \id_A$ and $i \circ p$ is homotopic to $\id_B$. The first composition yields by functoriality $\smash{p_V^G \circ i_V^G = \id_{A_V^G}}$ and this proves that $i_V^G$ is injective in homology. The second fact that $i \circ p \sim \id_B$ yields, by Proposition B.2 in \cite{BeKhRa}, an explicit homotopy $h: B \to B \otimes \Omega$ between $i \circ p = h(0)$ and $\id_B = h(1)$. By Lemma~\ref{lem:polhom} there is a homotopy $h_V^G$ between $\smash{(i\circ p)_V^G = i_V^G\circ p_V^G }$ and $\smash{\id_{B_V^G}}$ and by Remark B.4.4 in \cite{BeKhRa} they induce the same maps at the level of homologies. This proves that $i_V^G$ is also surjective in homology.
\item [(b)] The total left derived functor $\L(-)_V^G$ obtained in the previous point sends the class of $\smash{\gamma A \in \Ho(\DGA_S)}$ to the class of $\smash{\gamma (QA)_V^G \in \Ho(\cDGA_k)}$ where $QA$ is a cofibrant resolution in $\DGA_S$. Moreover $\HH_\bullet(-,V)^G$ maps $\gamma A$ to $\HH_\bullet ([\gamma (QA)_V])^G$. But because $G$ is reductive (and $k$ is a field of characteristic $0$) there is an isomorphism
\[
\HH_\bullet (\gamma [(QA)_V^G] ) \cong [\HH_\bullet(\gamma(QA_V))]^G\, .
\]
and this concludes the proof.
\end{itemize}
\end{proof}

An analogous result holds also for the functor restricted on non-negatively graded objects, and it can be actually obtained as a corollary of Theorem~\ref{thm:derivedV}:

\begin{Cor}
\label{cor:derivedV+}
\begin{itemize}
\item [(a)] $(-)_V^G : \DGA_S^+ \to \cDGA_k^+$ has a total left derived functor $\L(-)_V^G$.
\item [(b)] For every $A \in \DGA_S^+$ there is a natural isomorphism:
\begin{equation}
\HH_\bullet [\L(A)_V^G]  \cong \HH_\bullet (A,V)^G\, .
\end{equation}
\end{itemize}
\end{Cor}

\begin{proof}
Using Brown's lemma we just need to prove that $(-)_V^G$ sends a trivial cofibrations between cofibrant objects $\smash{A \accof B}$ to weak equivalences. We consider the following commutative diagram:
\begin{equation}
\begin{tikzcd}
  \DGA_S^+ \arrow[r, "(-)_V^G"] \arrow[d, "\iota"]
    & \cDGA_k^+ \arrow[d, "\iota" ] \\
  \DGA_S \arrow[r, "(-)_V^G" ]
&\cDGA_k \end{tikzcd}
\end{equation}
and observe that $\smash{\iota(A \accof B)}$ is still a trivial cofibration between cofibrant objects in $\DGA_S$, according to Remark~\ref{rem:+}. Now we can use the proof of Theorem~\ref{thm:derivedV} to conclude that the functor $(-)_V^G :\DGA_S \to \cDGA_k$ sends this map to a weak equivalence:
\[
(\iota A)_V^G=\iota (A_V^G) \overset{\sim}{\to} (\iota B)_V^G=\iota (B_V^G ) \in \cDGA_k\, .
\]
Finally from the very construction of $\iota$ we have that this map is a weak equivalence if and only if the map $\smash{A_V^G\overset{\sim}{\to} B_V^G\in \cDGA_k^+}$ is a weak equivalence. This concludes the proof of (a), while (b) follows from (a) as in Theorem~\ref{thm:derivedV}.
\end{proof}
Now we derive also the other isotypical components of the representation functor. Let us fix any irreducible, finite-dimensional representation $U_\lambda$ of the reductive group $G$. We consider the following functor:
\begin{equation}
\label{lambdasubf}
\begin{aligned}
(-)_{\lambda,V}^G :\, & \DGA_S \to \dgVect_k\\
&A \longmapsto \big(U_\lambda^*\otimes_k A_V\big)^G\, .
\end{aligned}
\end{equation}
which is the invariant subfunctor of the functor $(-)_{\lambda,V} := U_\lambda^* \otimes_k (-)_V$. Then we can prove the following analogue to Theorem~\ref{thm:derivedV}:
\begin{Thm}
\label{thm:derivedlambdaV}
\begin{itemize}
\item [(a)] The functor \eqref{lambdasubf} has a total left derived functor $\L (-)_{\lambda,V}^G$.
\item [(b)] For every $A \in \DGA_S$ there is a natural isomorphism:
\begin{equation}
\HH_\bullet [\L (A)_{\lambda,V}^G]  \cong \big( U_\lambda^* \otimes \HH_\bullet (A,V) \big)^G\, .
\end{equation}
\end{itemize}
\end{Thm}
To prove it we need the following analogue of Lemma \ref{lem:polhom}:
\begin{Lem}
\label{lem:polhom2}
Let $h:A \to B\otimes \Omega \in \DGA_S$ be a polynomial homotopy between $f,g:A \to B$. Then:
\begin{enumerate}
\item There is a homotopy $h_{\lambda,V} :A_{\lambda,V} \to B_{\lambda,V} \otimes \Omega \in \dgVect_k$ between $h_{\lambda,V}(0) = f_{\lambda,V}$ and $h_{\lambda,V}(1)=g_{\lambda,V}$.
\item $h_{\lambda,V}$ restricts to a morphism $h_{\lambda,V}^{G}: A_{\lambda,V}^G \to B_{\lambda,V}^G \otimes \Omega \in \dgVect_k$.
\end{enumerate}
\end{Lem}
\begin{proof} It is essentially a corollary of Lemma \ref{lem:polhom}. In fact, we can define $h_{\lambda, V}$ to be 
\[
h_{\lambda,V }: A_{\lambda,V} = U_\lambda^* \otimes A_V \xrightarrow[]{\id_{U_\lambda^*} \otimes h_V} U_\lambda^*\otimes B_V \otimes \Omega = B_{\lambda,V} \otimes \Omega\, ,
\]
where $h_V$ is the map from part (1) of Lemma \ref{lem:polhom}. The map $h_V$ was $G$-equivariant, and therefore also $h_{\lambda,V}=\id_{U_\lambda^*}\otimes h_V$, from which part (2) follows.
\end{proof}
\begin{proof}[Proof of Theorem \ref{thm:derivedlambdaV}] The proof works exactly as the proof of Theorem \ref{thm:derivedV}, using Lemma \ref{lem:polhom2} instead of Lemma \ref{lem:polhom}.
\end{proof}
The analogous results in the non-negative case also hold:
\begin{Cor}
\label{cor:derivedlambdaV+}
\begin{itemize}
\item [(a)] The functor $(-)_{\lambda,V}^G :\DGA_S^+ \to \dgVect_k^+$ has a total left derived functor $\L (-)_{\lambda,V}^G$.
\item [(b)] For every $A \in \DGA_S^+$ there is a natural isomorphism:
\begin{equation}
\HH_\bullet [\L (A)_{\lambda,V}^G]  \cong \big( U_\lambda^* \otimes \HH_\bullet (A,V) \big)^G\, .
\end{equation}
\end{itemize}
\end{Cor}
\begin{proof} The proof follows from Theorem~\ref{thm:derivedlambdaV} in the same way as the proof of Corollary~\ref{cor:derivedV+} followed from Theorem~\ref{thm:derivedV}.
\end{proof}

\subsection{K-theoretic classes} 
\label{subsec:k-theoreticclasses}

We use the classical $G$-invariant subfunctor $(-)_V^G: \Alg_S \to \cAlg_k$ to define
\begin{Defn}
The \emph{partial character scheme} of an algebra $A \in \Alg_S$ in a vector space $V$, relative to a subgroup $G \subset G_S$, is the affine quotient of the representation scheme:
\begin{equation}
\Rep_V^G(A):= \Rep_V(A)\sslash G = \Spec(A_V^G) \in \Aff_k\, .
\end{equation}
\end{Defn}
The name is motivated by the fact that in the absolute case $S=k$ and $G=\GL(V)$ the full group, we would obtain the classical scheme of characters $\smash{\Rep_V^{\GL(V)}(A)}$. The derived version is:
\begin{Defn}
The \emph{derived partial character scheme} of $A \in \Alg_S$ in a vector space $V$, relative to a subgroup $G \subset G_S$, is the affine quotient of the derived representation scheme:
\begin{equation}
\DRep_V^G(A):= \DRep_V(A)\sslash G = \RR\Spec\big(\L (A )_V^G \big) \in \Ho(\DGAff_k)\, .
\end{equation}
\end{Defn}

Let us recall that the obvious inclusion $\smash{\Sch_k \to \DGSch_k}$ has for right adjoint the truncation functor $\pi_0 :\DGSch_k \to \Sch_k$ that associates to a dg-scheme $X=(X_0,\O_{X,\bullet})$ the closed subscheme $\smash{\pi_0(X):=\Spec(\HH_0(\O_{X,\bullet}))\subset X_0}$:
\begin{equation}
\adjunct{\Sch_k}{\DGSch_k}{}{\pi_0}\, .
\end{equation}
Because the differential $d:\O_{X,i} \to \O_{X,i-1}$ is $\O_{X_0}$-linear, the homologies $\HH_i(\O_{X,\bullet})$ are quasicoherent sheaves on $X_0$, and also on the closed subscheme $\pi_0(X) \subset X_0$.  We can put these data together in a dg-affine scheme:
\[
X_h := \big( \pi_0(X), \HH_\bullet( \O_{X,\bullet}) \big) \in \DGAff_k\, ,
\]
which in the affine case $X =\Spec(A)$ is nothing but $\Spec(\HH_\bullet(A))$.
\begin{Defn} [Definition 2.2.6. in \cite{CiKa2}] 
\label{def:finitetype}
A dg-scheme $X$ is \emph{of finite type} if $X_0$ is a scheme of finite type and each $\O_{X,i}$ is a coherent sheaf on $X_0$.
\end{Defn}

Let now come to the case of our interest, a dg-affine scheme of finite type $X=\Spec(B)$, for which the sheaves $\HH_i(\O_{X,\bullet})$ are coherent both over $X_0$ and over $\pi_0(X)= \Spec(\HH_0(B))$, therefore they define a class in the algebraic K-theory\footnote{By algebraic K-theory of a scheme we mean the Grothendieck ring of the abelian category of coherent sheaves on it.}
\begin{equation}
[\HH_i(\O_{X,\bullet})] \in K (\pi_0(X))\, .
\end{equation}
We first consider the derived scheme $X=\DRep_V(A)$. Let us assume that $A$ is an algebra such that, for each vector space $V$, the following two conditions are satisfied:
\begin{enumerate}
\item The derived representation scheme $X=\DRep_V(A)$ is of finite type.
\item The structure sheaf $\O_{X,\bullet}$ of the derived representation scheme is bounded, in the sense that $\O_{X,i}=0$ for $i \gg 0$. 
\end{enumerate}
This is true for all algebras that we consider in this article, as we show in \S~\ref{subsec:cofres} and \S~\ref{subsec:zerolocuskoszul}. The truncated scheme obtained from the derived representation scheme is the classical representation scheme, as explained in Remark~\ref{rmk:zerohom}:
\[
\pi_0(\DRep_V(A)) = \Rep_V(A)\, .
\]
By condition (1) each homology defines a coherent sheaf on $\pi_0(X)=\Rep_V(A)$ and therefore a class
\[
\big[ \HH_i(A,V) \big] \in K(\Rep_V(A)) \, .
\]
By condition (2) there is only a finite number of them nonzero, therefore in particular the following definition makes sense, because the sum in \eqref{euldrep} is bounded:
\begin{Defn}
The \emph{virtual fundamental class} - or Euler characteristic of the derived representation scheme $X=\DRep_V(A)$ is the following invariant in the $K$-theory of the classical representation scheme:
\begin{equation}
\label{euldrep}
[X]^\vir= \chi(A,V):= \sum\limits_{i=0 }^\infty (-1)^i \big[ \HH_i(A,V) \big] \in K(\Rep_V(A))=K(\pi_0(X))\, .
\end{equation}
\end{Defn}
This virtual fundamental class carries an action of the group $G$, which is reductive, and therefore it decomposes into a direct sum of its irreducible components. To formalise this we first consider the quotient by derived partial character scheme $\smash{X^G= \DRep_V^G(A)}$, whose truncation is $\pi_0(X^G) = \Rep_V^G(A)$. For each finite-dimensional irreducible representation $U_\lambda$ of $G$ we proved the existence of the derived functor of the corresponding component $\smash{\L(-)_{\lambda,V}^G: \Ho(\DGA_S^+) \to \Ho(\dgVect_k^+)}$ and observed that
\[
\HH_i \big( \L(A)_{\lambda,V}^G \big) \cong \big(U_\lambda^* \otimes \HH_i(A,V) \big)^G \in \Mod_{\HH_0(A,V)^G}\, ,
\]
and therefore they define coherent sheaves on $\Rep_V^G(A)$.
\begin{Defn}
The Euler characteristic of the $U_\lambda$-irreducible component of the derived partial character scheme is
\begin{equation}
\chi^\lambda(A,V) := \sum\limits_{i=0}^\infty (-1)^i [\HH_i(\L(A)_{\lambda,V}^G)] \in K(\Rep_V^G(A))\, .
\end{equation}
\end{Defn}
We observe that the irreducible component corresponding to the trivial representation $U_0=k$ is the virtual fundamental class of the derived partial character scheme, which we denote by
\[
\chi^G(A,V) =  \sum\limits_{i=0}^\infty (-1)^i [\HH_i(A,V)^G)] = [X^G]^\vir \in K(\Rep_V^G(A))\, .
\]

\subsection{T-equivariant enrichment}
\label{subsec:Tenrich}

So far we have worked only with a group $G\subset G_S \subset \GL(V)$ that acts on the representation scheme $\Rep_V(A)$ because of the standard action on the vector space $V$. However, often the algebra $A$ itself comes with an action of some algebraic torus $T$ which helps when calculating its invariants (for example the corresponding decomposition of $A$ might consist of finite dimensional weight spaces, allowing a graded dimensions count). In this section we explain how such an action $T \acts A$ induces a well-defined group scheme action $T \acts \DRep_V(A)$, in the sense that different models for the derived representation scheme are linked by $T$-equivariant quasi-isomorphism, and therefore their homologies (and all the other invariants, as the Euler characteristics introduced in \S~\ref{subsec:k-theoreticclasses}) carry a well-defined induced $T$-action. 

First we give a notion of a \emph{rational} $T$-action, for an algebraic group $T \in \Grp_k$ on any (dg,commutative) algebra. 

\begin{Defn}
\label{def:rationalaction}
Let $\CC$ be any of the following categories: $\smash{\dgVect_k, \DGA_S,\cDGA_k}$ or their non-negatively graded versions. A \emph{rational action} of an algebraic group $T$ over $k$ on an object $A \in \CC$ is a morphism of groups $\smash{\rho : T \to \Aut_\CC(A)}$ with the additional property that every element $a \in A$ is contained in a finite dimensional $T$-stable vector subspace $a \in V \subset A$ on which the induced action $T \to \GL_k(V)$ is a morphism of algebraic groups over $k$. We denote by $\CC^T$ the category with objects the objects in $\CC$ with a rational $T$-action and morphisms the equivariant morphisms.
\end{Defn}

This definition is motivated by the fact that the equivalence of categories \eqref{dgaff} enriches to an equivalence of categories between $\smash{(\cDGA_k^+)^T}$ and the (opposite) category of dg-affine schemes with a group scheme action of $T$. 

\begin{Remark} 
\label{rem:enrich}
If we denote by $\Tcat$ the one-object groupoid associated to the group $T$, then a rational action on an object in $\CC$ is just a functor $\Tcat \to \CC$ with some additional properties, and a $T$-equivariant morphism is a natural transformation of functors. Another way to say this is that we can view the category $\smash{\CC^T \subset [\Tcat,\CC] }$ as a full subcategory of the category of functors. If $\CC,\DD$ are two among the categories mentioned in~\ref{def:rationalaction}, and $F:\CC \to \DD$ is any functor between them, then we can consider the induced functor on the functor categories $F_*=F\circ (-):[\Tcat,\CC] \to [\Tcat,\DD]$. If this induced functor sends objects of $\CC^T\subset [\Tcat,\CC]$ into objects of $\DD^T$, then it restricts to a functor that we denote by $F^T :\CC^T \to \DD^T$. This is true whenever $F$ is defined purely in ``algebraic terms''\footnote{We leave intentionally this as an intuitive, not well-defined, notion.}, which is the case of all the functors we considered so far. The induced functor $F^T$ is an \emph{enrichment} of the functor $F$ in the sense that we can recover $F$ under the natural forgetful functors:
\begin{equation}
\label{enrich}
\begin{tikzcd}
  \CC^T  \arrow{r}{F^T} \arrow{d}{U}
    & \DD^T \arrow{d}{U}  \\
 \CC \arrow{r}{F}
&\DD 
 \end{tikzcd}
\end{equation}
\end{Remark}

It is easy to see from the definition of the representation functor that a rational action $T \acts A$ induces (as explained in Remark~\ref{rem:enrich}), an action $T \acts A_V$ which is still rational, and therefore a group scheme action $T\acts \Rep_V(A)$. To summarise the adjunction \eqref{quillenpair0} enriches to an adjunction:
\begin{equation}
\adjunct{\big(\DGA_S^+\big)^T}{\big(\cDGA_k^+\big)^T}{(-)_V}{\End(V)\otimes_k(-)}\, .
\end{equation}
We did not add a superscript $(-)^T$ to the enriched functors in this diagram because we want to avoid confusion with the same symbols used with a different meaning in \S~\ref{subsec:ginvariants}. 

From now on we restrict ourselves to the case of our interest in this paper of an algebraic torus $T =(k^\times)^r$. To do what we promised to do in the beginning of this section we need to prove that, roughly speaking, any $T$-equivariant algebra admits an equivariant cofibrant replacement in the model category $\DGA_S^+$, and that any two such equivariant cofibrant replacements produce quasi-isomorphic representation schemes. To do it we introduce a model structure on the category $\smash{(\DGA_S^+)^T}$ compatible with the model structures on $\DGA_S^+$ under the forgetful functor (in the following Theorem we explain in which sense these model structures are compatible). We recall that $\DGA_S^+$ is equipped with the projective model structure in which weak equivalences are quasi-isomorphisms and fibrations are surjections in positive degrees. We also observe that actually the category of $T$-equivariant dg-algebras over $S$ is $\smash{(\DGA_S^+)^T = S \downarrow (\DGA_k^+)^T}$ nothing else but the under category of $T$-equivariant dg-algebras over $k$ receiving a map from $S$ if we give $S$ the trivial action, and therefore we only need to give a model structure in the absolute case $S=k$.

\begin{Thm}
\label{thm:T-homtheory}
There exists a model structure on $\smash{(\DGA_k^+)^T}$ with the following properties:
\begin{enumerate}
\item Weak equivalences / fibrations are exactly the maps that are weak equivalences / fibrations under the forgetful functor $U: (\DGA_k^+)^T \to \DGA_k^+$ (and cofibrations are the maps with the left-lifting property with respect to acyclic fibrations defined in this way).
\item The forgetful functor preserves cofibrations.
\end{enumerate}
\end{Thm}
\begin{proof} We refer the reader to Appendix~\ref{app:A} for the proof of this Theorem.
\end{proof}

As a corollary of this result we can naturally equip the derived representation scheme of a $T$-equivariant algebra with a group scheme action of $T$. In fact, let $S \in \Alg_k$ and $\smash{(A \in \Alg_S)^T=S \downarrow (\Alg_k)^T}$ be a $T$-equivariant algebra. 

\begin{Cor}
There is a well-defined action $T\acts \DRep_V(A)$ which is compatible with the one on $\smash{\Rep_V(A) \cong \pi_0(\DRep_V(A))}$ induced by $T\acts A$.
\end{Cor}
\begin{proof}
First of all, we can pick up a $T$-equivariant cofibrant replacement $\smash{Q \acfib A \in (\DGA_S^+)^T}$ using the model structure we just defined. Because of Theorem~\ref{thm:T-homtheory} (1) and (2), when we forget the $T$-action we still have a cofibrant replacement for $A$, therefore we can use this $Q$ as a model for $\DRep_V(A) = \Rep_V(Q)$. There is a natural $T$-action on this dg-scheme induced by $T \acts Q$, which is compatible with the one on its truncation $\smash{\pi_0(\Rep_V(Q)) \cong \Rep_V(A)}$.

To prove that the previous definition is well posed, we show that if $\smash{Q' \acfib A}$ is any another $T$-equivariant cofibrant replacement, then there is a $T$-equivariant quasi-isomorphisms of dg-schemes $\Rep_V(Q) \we \Rep_V(Q')$. In fact by the general machinery of model categories we can lift the identity map $1_A:A \to A $ to a $T$-equivariant (weak equivalence) between the two cofibrant replacements $f : Q \we Q'$. When we forget the $T$-action, this is still a weak equivalence, therefore giving an isomorphism $\gamma f$ in the homotopy category $\Ho(\DGA_S^+)$ and therefore $\L( \gamma f )_V$ is an isomorphism in $\Ho(\cDGA_k^+)$. But because both domain and codomain are cofibrant, $\L(\gamma f)_V = \gamma f_V$, and therefore $f_V : Q_V \to (Q')_V$ is a $T$-equivariant isomorphism of commutative dg-algebras, which dually gives the desired $T$-equivariant map $\Rep_V(Q)\we \Rep_V(Q')$.
\end{proof}

As a final consequence, the representation homology of a $T$-equivariant algebra, and all the other invariants defined in \S~\ref{subsec:k-theoreticclasses}, enrich to $T$-equivariant invariants. For example we can define the $T$-equivariant virtual fundamental class of the derived representation scheme $X=\DRep_V(A)$ as the following object in the equivariant $K$-theory of the classical representation scheme:
\begin{equation}
\label{Tvir}
[X]^\vir = \chi_T(A,V) := \sum\limits_{i =0}^{\infty} (-1)^i [ \HH_i(A,V)] \in K_T(\Rep_V(A))\, ,
\end{equation}
and also all the other $U_\lambda$-irreducible components for a reductive group $G$ by which we take the quotient (see \S~\ref{subsec:k-theoreticclasses}) as
\begin{equation}
\chi^\lambda_T(A,V) := \sum\limits_{i =0}^{\infty} [ \HH_i(\L(A)_{\lambda,V}^G)] \in K_T(\Rep_V^G(A)) = K_T(\Rep_V^G(A)\, .
\end{equation}
In particular for $U_0=k$ the trivial representation, we obtain an equivariant version of the virtual fundamental class of the derived partial character scheme $X^G=\DRep_V^G(A)$, which we denote by:
\begin{equation}
\chi^G_T(A,V) = \sum\limits_{i =0}^{\infty} [ \HH_i(A,V)^G] = [X^G]^\vir \in K_T(\Rep_V^G(A)) \, .
\end{equation}


\section{The case of Nakajima quiver varieties}
\label{sec:dgrepschemesfornakajma}
In this section we first recall the construction of Nakajima quiver varieties and secondly we construct some derived representation schemes related to them.


\subsection{Nakajima quiver varieties}
\label{subsec:generalitiesonnakajimaquivervarieties}
We already recalled in Example~\ref{ex:pathalg} that a finite quiver is a finite directed graph defined by its sets of vertices and edges $Q=(Q_0,Q_1)$ with two maps (source and target of an arrow) $s,t: Q_1 \to Q_0$. 

We first frame the quiver, this means that we add a new vertex for each old one with a new arrow from the new to the old. Then we double the framed quiver, in order to obtain a cotangent (symplectic) space when we consider its representations. We denote this quiver by $\smash{\overline{Q^{\fr}}}$.
\begin{figure} [htbp]   
\centering
\scalebox{0.9}{
\begin{tikzpicture}
\node (vertex0){ $\,\,\,\,\bullet \, \bm v$};
\path (vertex0) edge [very thick,->,out=135,in=45,looseness=5] node[above] {{\color{black}$\bm x$}} (vertex0);
\node [below= of vertex0, below=0.2](script) {Jordan quiver};
\node [right= of vertex0, right= 3cm](vertex){ $\,\,\,\,\bullet \, \bm v$};
\node [right= of vertex0, right= 0.3cm] (dummy0) {};
\node [left= of vertex, left= 0.3cm] (dummy) {};
\path (dummy0) edge[ ->, decorate, decoration=snake] (dummy);
\path (vertex) edge [very thick,->,out=135,in=45,looseness=5] node[above] {{\color{black}$\bm x$}} (vertex);
\path (vertex) edge [very thick,mygreen,->,out=35, in=145, looseness=9] node[above] {{\color{mygreen}$\bm y$}} (vertex);
\node [below= of vertex] (wertex) {{\color{myred} $\,\,\,\,\,\,\,\square\, \bm w$}};
\path (wertex) edge [very thick,myred, ->, out=90, in=270] node[right] {{\color{myred}$\bm i$}} (vertex);
\path (vertex) edge [very thick,mygreen, ->, out=230, in=130] node[left] {{\color{mygreen}$\bm j$}} (wertex);
\end{tikzpicture}
}
\caption{Example: {\color{myred} framing} and {\color{mygreen}doubling} the Jordan quiver. The framed vertices are usually denoted by a square symbol.}
\label{example}
\end{figure}
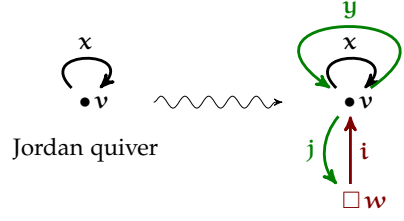
To consider representations of a framed (doubled) quiver, we need to fix two dimension vectors $\vv,\ww \in \N^{Q_0}$, and usually one assumes that (at least one of the components of) the framing vector is nonzero: $\ww\neq0$. 
\begin{Notation}
We denote the linear representations of the doubled, framed quiver by 
\begin{equation}
\label{repquiv}
M(Q,\vv,\ww):=L\left(\overline{Q^\fr},\vv,\ww\right) \cong \TT^*L\big(Q^\fr,\vv,\ww\big)\, .
\end{equation}
\end{Notation}
Explicitely it is the following cotangent linear space:
\begin{equation}
\begin{aligned}
M(Q,\vv,\ww) =\TT^* \left( \bigoplus\limits_{\gamma \in Q_1} \Hom_\C(\C^{v_{s(\gamma)}}, \C^{v_{t(\gamma)}})  \oplus \bigoplus_{a \in Q_0} \Hom_\C(\C^{w_a},\C^{v_a}) \right)\, .
\end{aligned}
\end{equation}
We denote elements of this space by quadruples $\smash{(X,Y,I,J)= (X_\gamma,Y_\gamma,I_a,J_a)_{\gamma,a}}$, where $X_\gamma, I_a \in L(Q^\fr,\vv,\ww)$ are elements of the representation space of the framed quiver, and $(Y_\gamma,J_a)$ are cotangent vectors to them. The gauge group is the general linear group on the set of vertices of the original quiver $Q$:
\begin{equation}
\label{Gv}
G=G_\vv := \prod\limits_{a \in Q_0} \GL_{v_a}(\C) \subset  \GL(\C^\vv\oplus \C^\ww)\, ,
\end{equation}
which acts by conjugation in a Hamiltonian fashion on $M(Q,\vv,\ww)$. The moment map for this action is 
\begin{equation}
\label{momentmap}
\begin{aligned}
\mu:  \,&M(Q,\vv,\ww) \to \lieg_\vv^* \cong \lieg_\vv \quad (\text{\small via trace})\\
&(X,Y,I,J) \longmapsto [X,Y]+IJ\, ,
\end{aligned}
\end{equation}
where in the above equation $[X,Y]+IJ$ is a shortened symbol for
\begin{equation}
\label{momentmap1}
[X,Y]+IJ= \left( \sum\limits_{\gamma: t(\gamma) =a} X_\gamma Y_\gamma - \sum\limits_{\gamma: s(\gamma)=a} Y_\gamma X_\gamma + I_a J_a \right)_{a\in Q_0} \in \bigoplus\limits_{a \in Q_0} \gl_{v_a}(\C) = \lieg_\vv\, .
\end{equation}
Nakajima varieties are defined as symplectic reductions of $M(Q,\vv,\ww)$ by this action. The affine Nakajima quiver variety is the geometric quotient:
\begin{equation}
\Mo(Q,\vv,\ww):= \mu^{-1}(0) \sslash G = \Spec \big( \O( \mu^{-1}(0))^G \big) \, .
\end{equation}
The GIT Nakajima variety is instead given by the choice of a character $\chi \in \Hom_{\Grp_\C}(G,\C^\times)$ as the proj of the graded ring of $\chi$-quasiinvariant functions on $\mu^{-1}(0)$:
\begin{equation}
\M^\chi(Q,\vv,\ww) = \mu^{-1}(0) \sslash_\chi G = \Proj \big( \O(\mu^{-1}(0))^{G,\chi} \big) 
\end{equation}
(elements of degree $n\geq 0$ of $\smash{\O(\mu^{-1}(0)^{G,\chi})}$ are functions $f\in \O(\mu^{-1}(0))$ with the property $\smash{f(g\cdot p) = \chi^n(g) f(p)}$ for all $g \in G$ and $p \in \mu^{-1}(0)$). The inclusion of $G$-invariant functions as degree zero elements of the graded ring of $\chi$-quasiinvariant functions $\O(\mu^{-1}(0))^G \subset \O(\mu^{-1}(0))^{G,\chi}$ induces a projective morphism:
\begin{equation}
p : \M^\chi(Q,\vv,\ww) \to \Mo(Q,\vv,\ww)\, ,
\end{equation}
which is often a symplectic resolution of singularities. Sometimes we denote these varieties simply by $\M^\chi,\Mo$ implicitly fixing the quiver $Q$, and the dimension vectors $\vv,\ww$.


\subsection{Derived representation schemes models} 
In Proposition~\ref{connection2} we showed how the linear space of representations of a quiver is isomorphic to the representation scheme for its path algebra. The same thing holds for the doubled, framed quiver so that
\begin{equation}
\label{linrep}
M(Q,\vv,\ww) = L\left( \overline{Q^\fr},\vv,\ww\right) \cong \Rep_{\vv,\ww} \left(\C \overline{Q^\fr}\right)\, .
\end{equation}
To obtain the zero locus of the moment map, we consider the $2$-sided ideal $\smash{\mathcal{I}_\mu \subset \C \overline{Q^\fr} }$ generated by the $|Q_0|$-elements of the path algebra described in~\eqref{momentmap}, and consider the quotient algebra
\begin{equation}
A := \C \overline{Q^\fr}  \slash \mathcal{I}_\mu \in \Alg_S\, ,
\end{equation}
relative to the subalgebra $S \subset A$ of idempotents, with fixed representation $\smash{\rho =\rho_{\vv,\ww} :S \to \End_\C(\C^\vv\oplus \C^\ww)}$ (as in~\eqref{rhov}). The following result is an immediate consequence of the fact that taking the quotient by some ideal amounts simply to impose these new relations in the representation scheme (see Examples~\ref{Exs1}.(6)):
\begin{Prop}
\label{prop:connection3}
The zero locus of the moment map $\mu$ is the (relative) representation scheme for the path algebra of the framed, doubled quiver, modulo the Hamiltonian relation:
\begin{equation} 
\label{mu0=rep}
\mu^{-1}(0) \cong \Rep_{\vv,\ww} \big(\C \overline{Q^\fr} \slash \mathcal{I}_\mu \big)\, .
\end{equation}
\end{Prop}
\begin{Notation}
We denote the corresponding derived representation scheme and representation homology by:
\begin{equation}
\begin{aligned}
&\DRep_{\vv,\ww} (A)=\Spec\big( \L\big(A \big)_{\vv,\ww} \big) \in \Ho(\DGAff_\C)  \, ,\\
&\HH_\bullet(A, \vv,\ww ) = \HH_\bullet\big(\L \big( A\big)_{\vv,\ww}  \big) \in \cDGA^+_\C\, .
\end{aligned}
\end{equation}
The representation homology $ \HH_\bullet(A,\vv,\ww)$ is a graded commutative algebra, so when we view it in $\cDGA^+_\C$ we mean that the differential is zero.
\end{Notation}
Remark~\ref{rmk:zerohom}, together with Proposition~\ref{prop:connection3} tells us that the $\pi_0$ of this derived scheme $X=\DRep_{\vv,\ww}(A)$ is the zero locus of the moment map:
\begin{equation}
\label{h0=mu0}
\pi_0(X) = \Spec \big( \HH_0\big(A,\vv,\ww \big) \big) \cong \mu^{-1} (0)\, .
\end{equation}
In particular when we consider the invariant subfunctor only by the gauge group on the original vertices $G$ \eqref{Gv}:
\begin{Cor}
\label{cor:derivedM0}
The $\pi_0$ of the partial character scheme $X^G=\DRep_{\vv,\ww}^G (A)$ is the affine Nakajima variety $\Mo$:
\begin{equation}
\label{hoG=mo}
\pi_0(X^G)= \pi_0(\DRep^G_{\vv,\ww}(A)) \cong  \Mo\, .
\end{equation}
\end{Cor}
\begin{proof} 
It follows directly from the previous observation \eqref{h0=mu0} and the Theorem~\ref{thm:derivedV}. More precisely:
\[
\pi_0(X^G) \cong \Spec \big( \HH_0 ( A, \vv,\ww )^G\big) \cong  \Spec \big( \O (\mu^{-1}(0))^{G} \big)  =\mu^{-1}(0) \sslash G = \Mo\, .
\]
\end{proof}

\subsection{K-theoretic classes in the affine Nakajima variety} 
\label{subsec:vfc}

In \S~\ref{subsec:cofres} we describe an explicit cofibrant resolution for our algebra $\smash{A=\C \overline{Q^\fr}}$, $\smash{A_\cof \acfib A }$ and therefore a model for the derived representation scheme $\smash{\DRep_{\vv,\ww}\big( A  \big)=  \Rep_{\vv,\ww} \big(  A_\cof \big)}$, but we can already use Corollary~\ref{cor:derivedM0} to define some interesting invariants in the K-theory of $\Mo=\Rep_{\vv,\ww}^G(A)$. Throughout this section we denote by $\smash{X=\DRep_{\vv,\ww}(A)}$ the derived representation scheme and by $\smash{X^G=\DRep^G_{\vv,\ww}(A)}$ the corresponding partial character scheme, whose $\pi_0(X^G)=\Mo$ is the affine Nakajima variety.

There is a torus, the (standard) maximal torus of the gauge group on the framing vertices $T_\ww \subset G_\ww$ acting on the linear space of representations $\Rep_{\vv,\ww}(A)$, and therefore as explained in \S~\ref{subsec:k-theoreticclasses} it induces an action $T_\ww \acts \DRep_{\vv,\ww}(A)$ and on its quotient by the gauge group $G_\vv$: $\smash{T_\ww \acts \DRep_{\vv,\ww}^{G_\vv}(A)}$. There is an additional ($2$-dimensional) torus $\smash{T_\hbar=(\C^\times)^2 \acts A}$ acting rationally on the path algebra of the doubled framed quiver. This action can be described by assigning, respectively, the following $\Z^2$-weights to the arrows $(x_\gamma,y_\gamma,i_a,j_a)$ (see \S~\ref{subsec:generalitiesonnakajimaquivervarieties} to recall the name of the arrows): $(1,0)$, $(0,1)$,$(1,1)$,$(0,0)$, or explicitly as $x_\gamma \mapsto \hbar_1 x_\gamma$, $y_\gamma \mapsto \hbar_2 y_\gamma $, $i_a \mapsto \hbar_1\hbar_2 i$, $j_a \mapsto j_a$. As explained in \S~\ref{subsec:Tenrich}, also this torus induces actions $\smash{T_\hbar \acts\DRep_{\vv,\ww}(A), \DRep_{\vv,\ww}^{G_\vv}(A)}$. In other words, the whole torus $T:=T_\ww \times T_\hbar $ acts on the derived representation scheme $\smash{X=\DRep_{\vv,\ww}(A)}$ and its partial character scheme $\smash{X^{G_\vv}= \DRep_{\vv,\ww}^{G_\vv}(A)}$.

Using the definitions we gave in \S~\ref{subsec:k-theoreticclasses} and \S~\ref{subsec:Tenrich} we obtain the following invariants in the (equivariant) $K$ theory of the affine Nakajima variety $\Mo= \Rep_{\vv,\ww}^{G_\vv}(A)$, for example the virtual fundamental class 
\begin{equation}
\label{vfc}
\big[X^{G_\vv} \big]^{\vir} = \sum\limits_{i =0}^{\infty}(-1)^i \big[ \HH_i (A, \vv,\ww)^{G_\vv}  \big] \in K_T\big(   \Mo \big) \, .
\end{equation}
More generally for each irreducible representation $U_\lambda $ of $G_\vv$, the Euler characterstic of the corresponding isotypical component as
\begin{equation}
\label{virtuallambda}
\chi_T^\lambda(A,\vv,\ww) =  \sum\limits_{i = 0}^{\infty} (-1)^i \big[ \HH_i(\L(A))_{\lambda,\vv,\ww}^{G_\vv}  \big] \in K_T\big( \Mo \big)\, .
\end{equation}


\subsection{Explicit cofibrant resolution}
\label{subsec:cofres}
In this section we describe an explicit cofibrant resolution for the $S$-algebra $A$ constructed in the previous Section. Let us recall that
\begin{equation}
\label{nakalg}
A = \C  \overline{Q^\fr}   \slash \mathcal{I}_\mu \in \Alg_S \hookrightarrow \DGA^+_S\, ,
\end{equation}
where $S$ is the subalgebra generated by the idempotents of the path algebra of the framed quiver. The main obstruction for this object to be cofibrant is the Hamiltonian relation described by the ideal $\mathcal{I}_\mu$. The simplest idea is then to add one more variable for each of the generating relations in $\mathcal{I}_\mu$ which kills the relation itself. This technique in general might not work due to higher homologies, but we prove that this case is one of the well-behaved cases.
We construct the following quiver $Q^\vartheta$, which is obtained by adding to the framed, doubled quiver $\smash{\overline{Q^\fr}}$, one loop called $\vartheta_a$ on each original vertex $a \in Q_0$. 

In the path algebra $\C Q^\vartheta$ we assign homological degree $0$ to the original arrows, and homological degree $1$ to the new arrows $\vartheta_a$. The differential is induced by the moment map (equations as in \eqref{momentmap1})
\[
d \vartheta_a = \mu_a (x,y,i,j) \, .
\]
We denote the resulting differential graded algebra by
\begin{equation}
\label{acof}
A_\cof := (\C Q^{\vartheta},d)  \in \DGA_\C^+\, .
\end{equation}
It sits in the following diagram
\begin{equation}
\label{cofres}
 \begin{tikzcd}[column sep=small]
& S \arrow[dl]{}{\iota} \arrow[dr]{}{\varphi} & \\
  A_\cof \arrow{rr}{\pi} &                         & A
\end{tikzcd}
 \quad
\end{equation}
where $\pi$ is the composition of the following two obvious projections:
\[
\pi: A_\cof \to (A_\cof)_0 = \C \overline{Q^{\fr} }  \to \C \overline{Q^\fr}  \slash \mathcal{I}_\mu = A\, .
\]
\begin{Thm}
\label{thm:cofrep}
$A_\cof$ is a cofibrant replacement for $A$ in $\DGA_S^+$.
\end{Thm}
 This amounts to prove that, in the diagram \eqref{cofres}, the map $\pi$ is an acyclic fibration, and $\iota$ is a cofibration.
\begin{Lem} 
\label{lem:acfib}
The map $\pi : A_\cof \to A$ is an acyclic fibration in $\DGA_\C^+ $.
\end{Lem}
\begin{proof}
We need to prove that:
\begin{itemize}
\item [(i)] $\pi$ is degreewise surjective in degrees $\geq 1$ (this is obvious, because $A$ is concentrated in degree $0$).
\item [(ii)] $\HH_i(\pi) : \HH_i (A_\cof) \to \HH_i(A)$ is an isomorphism for each $i \geq 0 $, which becomes proving that
\[
\begin{cases}
\HH_0(\pi) : \HH_0(A_\cof) \xrightarrow{\sim} A \, ,\\
\HH_i (A_\cof) =0, \quad  i\geq 1 \, .
\end{cases}
\]
\end{itemize}
$\HH_0(\pi)$ is an isomorphism, this is evident from the construction of $A_\cof$. We are left to prove that $A_\cof$ has no higher homologies.

We first decompose $A_\cof$ as a direct sum of the subalgebras of paths starting and finishing at a fixed couple of vertices:
\begin{equation}
\label{pathsdecomposition}
A_\cof = \bigoplus\limits_{a,b \in Q^\vartheta_0} P_{a,b}, \qquad P_{a,b} = \big\{\text{paths in $Q^\vartheta$ starting at $a$ and ending at $b$}\big\} \, .
\end{equation}
This decomposition preserves the differential, so we only need to prove that each $P_{a,b}$ has no higher homologies.

\underline{Claim:} If we substitute each $j_a,i_a$ with the cycle $c_a=i_a j_a$ and prove that the resulting dg-algebras have no higher homologies, then neither $P_{a,b}$ have.
\begin{proof}[Proof of the claim:]  Let us call $\widetilde{P}_{a,b}$ the dg-algebra of paths from $a$ to $b$ in the quiver $\smash{Q^\vartheta}$, where we substitute each pair of arrows $j_a,i_a$ with the cycle $c_a=i_aj_a$. Then, for each fixed $a,b \in Q_0$, we have four cases\footnote{We use the following notation, for a vertex $a \in Q_0$ in the original quiver, we denote by $\overline{a}$ the associated framed vertex.}: 
\begin{equation}
\label{casespaths}
\begin{cases}
P_{a,b} = \widetilde{P}_{a,b} \\
P_{a,\overline{b}} = j_b \cdot \widetilde{P}_{a,\overline{b}} \cong \widetilde{P}_{a,\overline{b}}\\
P_{\overline{a},b} = \widetilde{P}_{\overline{a},b} \cdot i_a \cong \widetilde{P}_{\overline{a},b}\\
P_{\overline{a},\overline{b}} = j_b \cdot \widetilde{P}_{\overline{a}, \overline{b} } \cdot i_a \cong  \widetilde{P}_{\overline{a}, \overline{b} }
\end{cases}
\end{equation}
\end{proof}
Now we consider the following filtration on the algebras $\widetilde{P}_{a,b}$: 
\[
F^p :=\Span_\C \{ \text{ paths with $\#x + \#y \geq 2p$} \} \, .
\]
Remember that the differential has the form ``$\smash{d \vartheta = [x,y] + c }$'', so that the associated graded has differential of the form $\smash{d_{\mathrm{gr}} \vartheta_a = c_a }$, which involves only loops on the vertices $a \in Q_0$. But then we can decompose the dg-algebras $\smash{\widetilde{P}_{a,b}}$ into their word structure\footnote{More precisely, we can decompose $\smash{\widetilde{P}_{a,b}}$ into the direct sum of those paths that, except for the arrows $\vartheta_a$ and $c_a$ --- meaning that we set these to $1$ --- are the same.}, and discover that the only non-trivial building blocks of which they are made of are dg-algebras of the form
\[
L = \big( k\langle \vartheta, c \rangle, d\vartheta = c \big) \,.
\]
which have no higher homologies\footnote{An elementary argument is to observe that the derivation defined by the formula $h(c) = \vartheta$ and $h(c)=0$, is a homotopy between the $0$ map and the map $\text{\emph{length}}(-) \cdot \rm{Id}$, which is an isomorphism in (homological) degrees $\geq 1$. This implies that $\rm{H}_i(L) = 0 $ for $i\geq 1$.}.
\end{proof}
\begin{Lem}
\label{lem:cof}
$\iota : S \to A_\cof$ is a cofibration in $ \DGA_\C^+$, or equivalently $A_\cof$ is a cofibrant object in $\DGA_S^+$.
\end{Lem}
\begin{proof}
We need to prove that $\iota$ has the left lifting property with respect to acyclic fibrations.
\begin{equation}
\label{lifting}
 \begin{tikzcd}
  S \arrow{d}{\varsigma} \arrow{r}
    &  B \arrow[twoheadrightarrow]{d}{\sim} \\
A_\cof \arrow{r} 
\arrow[dashrightarrow]{ru}{\exists}
& C \end{tikzcd}
\quad 
\end{equation}
Let us observe that because $A_\cof = \C Q^\vartheta$ is the (dg) path algebra of a quiver with idempotents $S$, we can view $A_\cof = T_S M:=S \oplus M \oplus (M\otimes_S M) \dots$ as the tensor algebra of the $S$-(dg)bimodule 
\[
M:= \Span_\C \big\{\text{arrows in} \,\, Q^\vartheta\big\}\, .
\]
But then find a lifting in the diagram \eqref{lifting} amounts to simply give a (linear) lifting of the (dg) vector space $M$, which is possible for the surjectivity of the map $B\acfib C$ (acyclic fibrations are surjective in every homological degree).

\end{proof}


\subsection{Koszul complex and complete intersections}
\label{subsec:zerolocuskoszul}
Theorem~\ref{thm:cofrep} tells us that a model for the derived representation scheme for the algebra $A$ is the representation scheme of the cofibrant replacement $A_\cof$. In this section we recognise it as the Koszul complex for the moment map, and in order to do so, we first recall a few notions and important classical results about the latter.

The Koszul complex can be thought of as one of the main examples of derived intersections of subschemes of a scheme. Classically, affine varieties are the simplest examples of intersections, being zero loci of some simultaneous polynomial equations $f_1,\dots, f_m \in \O(\A^n_\C) = \C[x_1,\dots,x_n] $:
\begin{equation} 
(X,\O) = \Spec \big( R\slash(f_1) \otimes_R \dots \otimes_R R\slash (f_m) \big), \qquad R=\C[x_1,\dots,x_n]\, .
\end{equation}
Then the associated \emph{derived intersection} can be defined as the derived scheme
\begin{equation}
(X,\O_\bullet) = \Spec \big(  R\slash(f_1) \otimes_R^{\L} \dots \otimes_R^{\L} R\slash (f_m) \big)\, ,
\end{equation}
where $\otimes_R^{\L}$ is the derived tensor product of $R$-modules. The algebra of functions on this derived scheme is the \textit{Koszul complex}:
\[
K = R\slash(f_1) \otimes_R^{\L} \dots \otimes_R^{\L} R\slash (f_m) \in \cDGA_\C^+\, .
\]
A more concrete way to describe it is the following: we can view the collection of functions $f=(f_1,\dots,f_m)$ as a map of affine schemes
\[
f: \A^n_\C \to V:= \A^m_\C \, ,
\]
and consider its dual map 
\[
\O(f): \O(V) = \Sym(V^*) \to \O(\A^n_\C)=R\, .
\]
Then the Koszul complex is the commutative dg-algebra $K=(R\otimes_\C\Lambda^\bullet (V^*),d)$, where $R$ is in homological degree $0$, the vector space $V^*$ is in homological degree $1$, and the differential  
\[
d:=\O (f)_{|V^*} : V^* \hookrightarrow \Sym(V^*) \to R \, .
\]
An useful classical result on the Koszul complex is  
\begin{Thm}[\cite{Ma}] 
\label{thm:koszulhomology}
The following are equivalent:
\begin{enumerate}
\item $\dim_\C( \Spec(\rm{H}_0(K))) = n - m$.
\item The sequence $f_1,\dots, f_m \in R$ is a regular sequence.
\item $\rm{H}_1(K) =0$. 
\item $\rm{H}_i(K)=0$ for all $i\geq 1$.
\end{enumerate}
\end{Thm}

Let us turn back to the case of our interest, in which we want to recognise 
\[
(A_\cof)_{\vv,\ww} =(\C Q^\vartheta)_{\vv,\ww}
\]
as the Koszul complex on the moment map. We recall that the quiver $Q^\vartheta$ is constructed from the quiver $\smash{\overline{Q^\fr}}$ by adding a new loop in homological degree $1$ on each of the original vertices of the quiver $Q$. Therefore, a representation of the path algebra $\C Q^\vartheta$ is just a representation of the subalgebra $\smash{\C \overline{Q^\fr}}$ (an element of the vector space $M(Q,\vv,\ww)$), together with a family of endomorphisms
\[\Theta = (\Theta_a)_{a\in Q_0}  \in \bigoplus\limits_{a \in Q_0} \gl_{v_a}(\C) = \lieg_\vv
\]
in homological degree $1$. Putting everything together we obtain
\[
(\C Q^\vartheta)_{\vv,\ww} = \O( M(Q,\vv,\ww))\otimes_\C \Lambda^\bullet \lieg_\vv \in \cDGA_\C^+\, ,
\]
which is nothing else but the Koszul complex for the zero locus defined by the moment map
\[
\mu: M(Q,\vv,\ww) \to \lieg_\vv^*\, .
\]
Its spectrum is a model for our derived representation scheme, as the derived intersection of the moment map equations:
\begin{Thm}
\label{thm:drep=koszul}
The cofibrant resolution $A_\cof \acfib A $ in $\DGA_S^+$ gives a model for the derived representation scheme as the (spectrum of the) Koszul complex on the moment map:
\begin{equation}
\label{drep=koszul2}
\DRep_{\vv,\ww} \big( A \big) \cong \Rep_{\vv,\ww} \big(  A_\cof) = \Spec \big( \O(M(Q,\vv,\ww))  \otimes \Lambda^\bullet \lieg \big)\, .
\end{equation}
\end{Thm}
In particular we can observe that this is a derived scheme of finite type (Definition~\ref{def:finitetype}) and that the Koszul complex is bounded, Therefore all the invariants defined in \S~\ref{subsec:vfc} (\eqref{vfc},\eqref{virtuallambda}) make sense, because the sums are bounded (by the dimension of the Lie algebra $\smash{\dim_\C \lieg_\vv = \vv^2=\vv\cdot \vv}$).

\begin{Remark} In \S~\ref{subsec:cofres} we gave a self-contained proof of why the resolution provided by the path algebra of the quiver $Q^\vartheta$ obtained by adding one loop on each vertex in which the corresponding component of the moment map is considered (i.e. the original vertices) works. In \S~\ref{subsec:zerolocuskoszul} we explained why the resulting representation scheme is the Koszul complex on the moment map. We remark that the same results can be explained in a slightly different flavour through the theory of noncommutative complete intersections (NCCI) and partial preprojective algebras (\cite{CrEtGi}, \cite{EtGi}).
\end{Remark}


\section{Main results}
\label{sec:comparisonresults}

\subsection{Flat moment map and vanishing representation homology}
\label{subsec:classicalresultsnaka}
In this section we recall some classical results on the flatness for the moment map of Nakajima quiver varieties which are useful for our purposes. We show how flatness is equivalent to the condition of vanishing of higher representation homologies for the corresponding algebra. 

Remember that for each quiver $Q$ and for each fixed dimensions $\vv,\ww \in \N^{Q_0}$ we have the corresponding Nakajima varieties $\Mo$ (affine) and $\M^\chi$ (quasiprojective), where $\chi \in \Hom_{\Grp_\C} (G_\vv, \C^\times) $ is a given (nontrivial) character. We also recall that the group of all characters of the gauge group $G=G_\vv=\prod_{a\in Q_0} \GL_{v_a}(\C)$ is isomorphic to the lattice
\[
\Z^{Q_0}  \cong \Hom_{\Grp_\C}(G,\C^\times)\, ,
\]
via the assignment
\[
\theta \mapsto \chi_\theta(g) = \prod\limits_{a \in Q_0} \det(g_a)^{\theta_a}\, .
\]
In this section we use the parameter $\theta$ for the characters and denote $\M^{\chi_\theta}$ simply by $\M^\theta$.

We recall that the \emph{Cartan matrix} of the quiver $Q$ is the matrix $\smash{C_Q =2 \cdot \Id - A_{\overline{Q}}}$, where $\smash{A_{\overline{Q}}}$ is the adjacency matrix of the doubled quiver $\overline{Q}$. For a fixed dimension vector $\vv \in \N^{Q_0}$, a vector $\theta \in \Z^{Q_0}$ is called \emph{$\vv$-regular}, if for each $\alpha \in \Z^{Q_0} \backslash \{ 0\}$ such that $C_Q \alpha \cdot \alpha \leq 2$ and $0 \leq \alpha \leq \vv$ (component-wise) then 
\[
\theta_1 \alpha_1 + \dots + \theta_{|Q_0|} \alpha_{|Q_0|} \neq 0 \, .
\]
The subset of $\mathbb{R}^{Q_0}$ of $\vv$-regular vectors is the complement of some hyperplanes. Its connected components are called \emph{chambers}, and the variety $\M^\theta$ depends only on the chamber of $\theta$.
\begin{Thm}[Theorem 5.2.2. in \cite{Gi1}] 
\label{thm:naka}
Let $\vv\in \N^{Q_0}$ be a dimension vector and  $\theta \in \Z^{Q_0}$ be $\vv$-regular, then any $\theta$-semistable point in $\mu^{-1}(0)$ is $\theta$-stable and $\M^\theta$ is a smooth, connected, variety of (complex) dimension
\[
\dim \M^\theta = 2 \vv\cdot \ww - C_Q \vv \cdot \vv \, ,
\]
(with the convention that $\M^\theta=\emptyset$ when this dimension is negative).
\end{Thm}
\begin{Remark}
\label{rmk:dim}
Observe that the dimension counting is what we would expect. In fact
\[
\dim (M(Q,\vv,\ww)) = 2 \vv\cdot \ww  +A_{\overline{Q}} \vv\cdot \vv = 2\vv \cdot \ww - C_Q 
\vv \cdot \vv + 2 \vv\cdot \vv \, .
\]
When we take the zero locus by $\mu$ we expect to decrease the dimension by the number of equations of $\mu$, which is $\vv \cdot \vv $ and then again by $\vv\cdot \vv$ when taking the $G_\vv$-quotient.
\end{Remark}
Let us consider for some $\vv$-regular $\theta$ the natural affinisation morphism
\begin{equation}
\varphi : \M^\theta \to \Spec\big( \O(\M^\theta) \big) \, .
\end{equation}
This morphism is a Poisson morphism\footnote{The Poisson structure on Nakajima varieties comes from the general formalism of Hamiltonian reduction, and coincides with the one induced by the symplectic form on the regular locus.} (obviously, because $\varphi^*$ is the identity) and it is a resolution of singularities (i.e. projective and birational) (\cite{BeLo}). The variety $\M^\theta$ depends, a priori on the chamber of $\theta$, but actually its affinisation $\Spec(\O(\M^\theta))$ is independent of the choice of $\vv$-regular $\theta$. We can call this variety simply $\M$ and we obtain a diagram of the following form
\begin{equation}
  \begin{tikzcd}
  \M^\theta \arrow{d}{\varphi} \arrow{rd}{p}  \\
 \M \arrow{r}{\psi}& \Mo
     \end{tikzcd}
\end{equation}
which is the so-called Stein factorisation (\cite{St}) of the proper morphism $p$. The pre-image of the point $0 \in \Mo$ through $\psi $ is always $0 \in \M$. In particular the fiber $p^{-1}(0)$ is equal to the central fiber $\varphi^{-1}(0)$ of the affinisation morphism
\[
p^{-1}(0) = (\psi \circ \varphi)^{-1}(0) = \varphi^{-1}(\psi^{-1}(0))= \varphi^{-1}(0)\, ,
\]
 and therefore is a homotopy retract of the variety $\M^\theta$. 
\begin{Thm}[\cite{BeLo}] 
\label{thm:nakalosevginz}
If the moment map $\mu : M(Q,\vv,\ww) \to \lieg_\vv^*$ is flat, then $\psi$ is an isomorphism, and in particular $\O(\M^\theta) \cong \O(\Mo)$.
\end{Thm}

The combinatorial criterium for the flatness of the moment map proved in \cite{Cr} is given in the setting of a non-framed quiver $\Gamma$. For any dimension vector $\balpha \in \N^{\Gamma_0}$ we consider the linear space of representations of the doubled quiver $L(\overline{\Gamma},\balpha)$. The gauge group acting a priori in a non-trivial way is now $\smash{G_\balpha \slash \C^\times}$
because, without the framing, the diagonal torus $\smash{\C^\times \subset G_\balpha}$ acts trivially on the linear space of representations. The Lie algebra of this group can be identified with the subalgebra $\smash{\lieg_\balpha^\n \subset \lieg_\balpha= \oplus_{i} \gl_{\alpha_i}(\C)}$ of matrices with sum of their traces equal to zero (the notation $\smash{\lieg_\balpha^\n}$ is borrowed from \cite{EtGi}). The moment map is now
\[
\begin{aligned}
\mu_\balpha : &L(\overline{\Gamma},\balpha) \to \lieg_\balpha^\n\\
& x \longmapsto [x,x^*]\, .
\end{aligned}
\]
Let us denote by $p$ the following function
\[
p: \N^{\Gamma_0} \to \Z \,, \qquad  p(\balpha) : =1+ \sum\limits_{\gamma \in \Gamma_1} \alpha_{s(\gamma)} \alpha_{t(\gamma)} - \balpha \cdot \balpha \, .
\]
\begin{Thm} [Theorem 1.1 in \cite{Cr}]
\label{thm:flatness}
The following are equivalent:
\begin{enumerate}
\item $\mu_\balpha$ is a flat morphism.
\item $\mu^{-1}_\balpha(0)$ has dimension $\balpha \cdot \balpha -1 + 2 p(\balpha)$ ($=\dim L(\overline{\Gamma},\balpha) - \dim \lieg_\balpha^\n$).
\item $p(\balpha) \geq \sum_{t=1}^r p(\bbeta^{(t)})$ for each decomposition $\balpha = \bbeta^{(1)} + \dots + \bbeta^{(r)}$ with each $\bbeta^{(t)} $ positive root.
\item $p(\balpha) \geq  \sum_{t=1}^r p(\bbeta^{(t)})$ for each decomposition  $\balpha = \bbeta^{(1)} + \dots + \bbeta^{(r)}$ with each $\bbeta^{(t)} \in \N^{\Gamma_0} \backslash \{0\}$.
\end{enumerate}
\end{Thm}
In a remark in \S~1 in \cite{Cr}, Crawley-Boevey explains how to adapt this setting to the situation of a framed quiver. From a quiver $Q$ and a framing vector $\ww$ we can construct a new quiver $\Gamma:=Q^{\infty}$, which is obtained by adding only one new vertex, denoted by $\infty$, together with a number of $w_a$ arrows towards each vertex $a \in Q_0$. If we fix now a dimension vector $\vv \in \N^{Q_0}$ and define the new vector $\balpha:= (\vv,1) \in \N^{\Gamma_0}$, then 
\begin{equation}
L(\overline{\Gamma}, \balpha) \cong L\big(\overline{Q^\fr}, \vv,\ww \big)= M(Q,\vv,\ww) \, ,
\end{equation}
by splitting the $v_a \times w_a$ matrices in $M(Q,\vv,\ww)$ in columns and the $w_a \times v_a$ matrices in rows. The two gauge groups are also isomorphic: $G_\balpha \slash \C^\times \cong G_\vv$, and under this isomorphism their actions on $L(\overline{\Gamma},\balpha) \cong M(Q,\vv,\ww)$ are the same. Therefore also the moment maps are identified:
\begin{equation}
\label{mudiag}
\begin{tikzcd}
  L(\overline{\Gamma},\balpha) \arrow{r}{\mu_\balpha} \arrow{d}{\sim}
    & \lieg_\balpha^\n \arrow{d}{\sim} \\
  M(Q,\vv,\ww) \arrow{r}{\mu}
&\lieg_\vv  \end{tikzcd}
\end{equation}
and we have the following criterium:
\begin{Cor} Consider the quiver $\smash{\overline{Q^\fr}}$ with dimension vectors $\vv,\ww \in \N^{Q_0}$, and the quiver $\Gamma=Q^\infty$ with $\balpha=(\vv,1)$. Then the following are equivalent:
\begin{enumerate}
\item $\mu$ is flat.
\item $\mu_\balpha$ is flat.
\end{enumerate}
\end{Cor}
For the condition (2) now we can use the combinatorical test given by Theorem~\ref{thm:flatness}, and using this result, we can prove that the derived representation scheme has vanishing higher homologies if and only if the moment map $\mu$ is flat:
\begin{Thm}
\label{thm:0}
The representation homology $\HH_\bullet(A,\vv,\ww)$ for the algebra~\ref{nakalg} vanishes if and only if the moment map $\mu$ is flat.
\end{Thm}
\begin{proof}
Because of the diagram \eqref{mudiag} the moment map $\mu$ is flat if and only if $\mu_\balpha$ is flat and by Theorem~\ref{thm:flatness}, condition (2), this happens if and only if 
\begin{equation}
\label{dim}
\dim \mu^{-1}(0) = \dim \mu_\balpha^{-1}(0) = \dim L(\overline{\Gamma},\balpha) - \dim \lieg_\balpha^\n = \dim M(Q,\vv,\ww)- \dim \lieg_\vv \, .
\end{equation}
The representation homology is the homology of the Koszul complex 
\[
\HH_\bullet(A,\vv,\ww) = \HH_\bullet \left( \O\big(M(Q,\vv,\ww) \big) \otimes \Lambda^\bullet \lieg_\vv\right)\, ,
\]
and therefore, by Theorem~\ref{thm:koszulhomology}, it vanishes in degrees $i\geq 1$ if and only if the dimension condition \eqref{dim} is satisfied.
\end{proof}

In the following examples we use Theorem~\ref{thm:flatness} for some quivers and we find the combinatorical condition on the dimension vectors for the moment map to be flat. It is convenient to observe that for the quiver $\Gamma =Q^\infty$ the map $p$ is, for vectors of the form $(\bbeta,1)$ or $(\bbeta,0)$ (that is the only type of vectors that we need to decompose the dimension vector $\balpha=(\vv,1)$):
\[
\begin{aligned}
&p(\bbeta,1) = \sum\limits_{\gamma \in Q_1} \beta_{s(\gamma)} \beta_{t(\gamma)}  + \bbeta \cdot \ww - \bbeta \cdot \bbeta \, , \\
&p(\bbeta,0) = p(\bbeta,1) + 1\, .
\end{aligned}
\]
\begin{Exs}
\label{Ex4}
\begin{enumerate}
\item 
\label{Exs4.1} 
The first example is that of a single-vertex quiver $Q=A_1$ with no arrows, whose $\Gamma=Q^\infty$ becomes a quiver with 2 vertices and $w$ arrows going from one to the other. We need to test for which $v$ it holds that for each decomposition
\[
(v,1) = (\beta_0, 1) +(\beta_1,0) +\dots +(\beta_r,0)\, ,\qquad \beta_t \geq 0\, .
\]
the following inequality holds:
\[
v(w-v)  \geq \beta_0(w-\beta_0) + r - \beta_1^2 - \dots - \beta_r^2\, .
\]
We can observe that actually all $\beta_1,\dots ,\beta_r \geq 1$ and therefore the function $r-\beta_1^2 -\dots -\beta_r^2 $ reaches its maximum for $\beta_1 = \dots = \beta_r=1$ for which it is $0$. So we just need to test that
\[
v(w-v) \geq \beta_0(w-\beta_0), \qquad \forall \beta_0=0, \dots, v-1
\]
The inequality can also be rewritten as
\[
\cancel{(v-\beta_0)} w \geq \cancel{(v-\beta_0)}(v+ \beta_0)\,, \quad \forall \beta_0=0, \dots, v-1 \qquad \Leftrightarrow \qquad w \geq 2v-1\, .
\]
\item The second example is a quiver with one vertex and $m$ loops ($m \geq 1$). In particular the Jordan quiver for $m=1$ described in Figure~\ref{example}. We show that for each choice of $v\geq0$ and $w\geq 1$ the moment map is flat. The quiver $\Gamma=Q^\infty$ still has 2 vertices, the first one with $m$ loops and $w$ arrows connecting the $2$\textsuperscript{nd} to the $1$\textsuperscript{st}, so that:
\[
p(\alpha_1,\alpha_2) =1 + m \alpha_1^2 + w \alpha_1 \alpha_2 - \alpha_1^2 -\alpha_2^2 \, .
\]
We need to test that for each decomposition
\[
(v,1) = (\beta_0,1) + (\beta_1,0) + \dots + (\beta_r,1) ,\qquad \beta_1,\dots,\beta_r \geq 1\, ,
\]
the following inequality holds
\[
(m-1)v^2 + vw \geq (m-1)\beta_0^2 + \beta_0 w + r +  (m-1)(\beta_1^2 + \dots + \beta_r^2) \,,
\]
which is actually true component-wise because
\[
\begin{cases}
(m-1)v^2 \geq (m-1)(\beta_0^2 + \dots +\beta_r^2)\,  \\
vw = (\beta_0 + \dots +\beta_r ) w \geq \beta_0 w + r w \geq \beta_0 w + r \, .
\end{cases}
\]
Therefore the moment map is always flat.
\item The third example is the quiver $Q=A_{n-1}$ with the following particular choice of vectors $\vv=(1,\dots,1)$ and $w_a=\delta_{a,1} + \delta_{a,n-1}$ (for which the Nakajima variety is the symplectic dual of $\TT^*\PP^{n-1}$, as explained in the next Section). The resulting quiver $\Gamma=Q^\infty$ is the cyclic quiver with $n$ vertices and dimension vector $\balpha=(1,\dots,1)$ constant to $1$, for which it is easy to check that the moment map is flat. In fact $p(\balpha)=1$ while for any other $\bbeta \in \N^n$, $0\neq \bbeta \neq \balpha$ we have $p(\bbeta)\leq 0$ so that condition (4) of Theorem~\ref{thm:flatness} is satisfied.
\end{enumerate}
\end{Exs}

\subsection{Kirwan map and tautological sheaves} 
\label{subsec:kirwan}

Let $\M^\chi=\M^\chi(Q,\vv,\ww)$ be a smooth Nakajima quiver variety (so $\chi=\chi_\theta$ with $\theta$ being $\vv$-regular, see Theorem~\ref{thm:naka}), then the locus of $\chi$-semistable points coincides with the locus of $\chi$-stable points, on which the action is free, and
\[
\M^\chi =\mu^{-1}(0)\sslash_\chi G=  \mu^{-1}(0)^{\chi\st}\slash G\, .
\]
The equivariant Kirwan map (in cohomology) is the map 
\begin{equation}
\label{Tkir}
\kappa_T :\HH^\bullet_{G\times T} \left( \mu^{-1}(0) \right) \to \HH^\bullet_{T} ( \M^\chi)\, ,
\end{equation}
obtained by composing the natural pullback for the inclusion $\smash{\mu^{-1}(0)^{\chi\st} \overset{\iota}{\subset} \mu^{-1}(0)}$ with the isomorphism $\smash{\HH_{G\times T}^\bullet\left(\mu^{-1}(0)^{\chi\st}\right) \cong \HH_{T}^\bullet(\M^\chi)}$ due to the fact that the $G$-action on the $\chi$-stable locus is free:
\[
\HH_{G\times T}^\bullet \left( \mu^{-1}(0) \right) \xrightarrow{\iota^\bullet} \HH_{G\times T}^\bullet \left(\mu^{-1}(0)^{\chi\st} \right) \cong \HH_{T}^\bullet\left( \mu^{-1}(0)^{\chi\st}\slash G \right) =\HH_T^\bullet(\M^\chi)\, .
\]
McGerty and Nevins have recently shown that the Kirwan map~\eqref{Tkir} is surjective (\cite[Corollary 1.5]{McNe}), and that the same holds for other generalised cohomology theories such as $K$-theory and elliptic cohomology. We are particularly interested in the K-theory, so the Kirwan map is
\begin{equation}
\label{TkirK}
\kappa_T: \KK_{G\times T}\left(\mu^{-1}(0) \right) \to \KK_{T}\left(\M^\chi\right)\, .
\end{equation}
Moreover the zero locus of the moment map $\mu^{-1}(0)$ is equivariantly contractible:
\[
\KK_{G\times T} \left(\mu^{-1}(0)\right) \cong \KK_{G\times T} (\pt) =\Rring(G\times T ) \cong\Rring(G)\otimes \Rring(T)\, ,
\]
where $\Rring(-)$ is the representation ring (over $\C$), so the Kirwan map has the form:
\begin{equation}
\label{TkirK2}
\kappa_T : \Rring(G)\otimes \Rring(T) \to  \KK_T\big( \M^\chi \big)\, ,
\end{equation}
and it is a surjective map of $\Rring(T)$-modules. $\KK_T(\M^\chi)$ is therefore generated by \textit{tautological classes}, because they come from classes of topologically trivial vector bundles: if $U$ is a $G\times T$-module, and $[U] \in \Rring(G\times T)$ is its class, then
\begin{equation}
\label{taut}
 \kappa_T([U]) =\big[ \big( \mu^{-1}(0)^{\chi\st} \times U\big) \slash G \big] \in \KK_T\big( \mu^{-1}(0)^{\chi\text{-st}}\slash G \big)= \KK_T(\M^\chi)\, .
\end{equation}
Moreover the map~\eqref{TkirK2} is a map of $\Rring(T)$-modules, so the only non-trivial part consist in its image on vector spaces $U$ that are only representations of $G$. For $U=V_\lambda$ irreducible representation of $G$, we denote by a calligraphic $\mathcal{V}_\lambda$ the sheaf whose K-theoretic class is $[\mathcal{V}_\lambda] = \kappa_T([V_\lambda]) \in \KK_T(\M^\chi)$. We can use these tautological classes to define invariants in the K-theory of the affine Nakajima variety by using the pushforward under the map $p$:
\begin{equation}
\label{kirnak}
\Rring(G) \otimes \Rring(T) \xrightarrow{\kappa_T} K_T( \M^\chi)  \xrightarrow{p_*} K_T( \Mo)\, .
\end{equation}
It is important to recall that in general the push-forward of a proper map $p$ in $K$-theory is given by the alternate sums of right-derived functors of $p_*$. In this particular case the target variety $\Mo$ is affine, therefore this alternate sum calculates the Euler characteristic of a sheaf $\mathcal{F}$ on $\M^\chi$, under the natural identifications:
\begin{equation}
\label{eul}
p_* ([\mathcal{F}]) =\chi_T(\M^\chi, \mathcal{F})  \in K_T(\O(\Mo)-\Mod) \cong K_T(\Mo)\, .
\end{equation}
The structure of $\O(\Mo)$-module comes from the fact that the cohomologies $\HH^i(\M^\chi,\mathcal{F})$ have a structure of $\O(\M^\chi)$-modules and the map $p:\M^\chi \to \Mo$ gives to the latter a structure of $\O(\Mo)$-module.

For an irreducible representation $U=V_\lambda$ of $G$ the composition~\eqref{kirnak} gives the Euler characteristic of the corresponding tautological sheaf $\mathcal{V}_\lambda$:
\begin{equation}
\label{eultaut}
p_*\big(\kappa_T([V_\lambda])\big) = p_* \big( [\mathcal{V}_\lambda] \big) = \chi_T(\M^\chi, \mathcal{V}_\lambda) \in K_T(\Mo)\, .
\end{equation}
The notable special case of $U=V_0$ the trivial $1$-dimensional representation of $G$, has image under the Kirwan map the ($K$-theoretic class of the) sheaf of functions on the GIT quotient $\mathcal{V}_0=\O_{\M^\chi}$, and its Euler characteristic:
\begin{equation}
\label{strmchi}
p_*(\kappa_T([V_0])) =p_*([\O_{\M^\chi}]) = \chi_T(\M^\chi,\O_{\M^\chi} )  \in K_T(\Mo)\, .
\end{equation}


\subsection{Comparison theorem and first integral formula}
\label{subsec:comparisontheorem}

In \S~\ref{subsec:vfc} we defined the virtual fundamental classes of the isotypical components of the derived character scheme
\begin{equation}
\label{drepgvfc}
\chi_T^\lambda(A,\vv,\ww)= \sum\limits_{i=0}^{\infty} (-1)^i \left[ \left( V_\lambda^* \otimes \HH_i(A,\vv,\ww)] \right)^G \right] \in K_T(\Mo)\, ,
\end{equation}
and in particular for $V_\lambda =V_0 =\C$: 
\begin{equation}
\label{drepgvfc00}
\chi_T^0(A,\vv,\ww) = \chi_T^G(A,\vv,\ww)= \sum\limits_{i=0}^{\infty} (-1)^i \left[  \HH_i(A,\vv,\ww)^G \right] \in K_T(\Mo)\, .
\end{equation}

\begin{Thm}
\label{thm:1}
Let $\vv,\ww$ be dimension vectors for which the moment map is flat, and let $\chi=\chi_\theta$ with $\theta$ $\vv$-regular, so that $\M^\chi(Q,\vv,\ww)$ is smooth. Then we have the following equality in the equivariant $\KK$-theory of the affine Nakajima variety :
\begin{equation}
\label{eq:thm1}
p_*([\O_{\M^\chi(Q,\vv,\ww)}]) = [\O_{\Mo(Q,\vv,\ww)}] =\chi_T^G(A,\vv,\ww)  \in K_T\left(\Mo(Q,\vv,\ww)\right)\, .
\end{equation}
\end{Thm}
\begin{proof}
The first equality is a somewhat classical result. Firstly the (derived) pushforward in $\KK$-theory coincides with the underived pushforward
\[
p_*([\O_{\M^\chi}]) = \chi_T(\M^\chi,\O_{\M^\chi} ) = \sum\limits_{i\geq 0} (-1)^i [ \HH^i (\M^\chi,\O_{\M^\chi}) ] =\big[ \O_{\M^\chi} \big] \in K_T(\Mo)\, ,
\]
because of the vanishing of higher cohomologies (Grauert-Riemenschneider theorem, (\cite{GrRi})). 
Moreover when the moment map is flat and $\M^\chi$ is smooth we can use Theorem~\eqref{thm:nakalosevginz}:
\[
 \big[ \O_{\M^\chi} \big] = [\O_{\Mo}]  \in K_T(\Mo)  \, 
\]
Finally by Theorem~\ref{thm:flatacyclic} the representation homology $\HH_\bullet(A,\vv,\ww)$ vanishes in positive degrees, so that the Euler characteristic of its $G$-invariant part~\eqref{drepgvfc00} is:
\[
\chi_T^G(A,\vv,\ww) = [\HH_0(A,\vv,\ww)^G] \overset{\text{Cor~\ref{cor:derivedM0}}}{=} [ \O_\Mo] \, .
\]
\end{proof}

\begin{Remark}
In light of the previous explanations that we gave during the course of the paper, the result stated in Theorem~\ref{thm:1} is not entirely surprising:
\begin{enumerate}
\item
On one hand we have a symplectic resolution of singularities $p:\M^\chi \to \Mo$ therefore it is expected that functions on the smooth variety $\M^\chi$ are equal to functions on the singular $\Mo$.
\item
On the other hand $\mu^{-1}(0)$ is a complete intersection in the linear space of representations $M(Q,\vv,\ww)$, therefore the Koszul complex $\smash{\O(\DRep_{\vv,\ww}(A)) \cong \O\left(M(Q,\vv,\ww)\right) \otimes \Lambda^\bullet \lieg_\vv}$ is a resolution of $\O(\mu^{-1}(0))$:
\begin{equation}
\label{unos}
\HH_i(A,\vv,\ww) =
\begin{cases}
\begin{aligned}
&\O(\mu^{-1}(0))\,, \quad &i=0\\
&\,0\,,&i \geq 1
\end{aligned}
\end{cases} 
\quad \left(\implies \quad \chi_T(A,\vv,\ww) = \O(\mu^{-1}(0)) \right)
\end{equation}
and the subcomplex of $G$-invariants is a resolution of the functions on $\Mo$:
\begin{equation}
\label{dues}
\HH_i(A,\vv,\ww)^G =
\begin{cases}
\begin{aligned}
&\O(\mu^{-1}(0))^G\,, \quad &i=0\\
&\,0\,,&i \geq 1
\end{aligned}
\end{cases} 
\quad \left(\implies \quad \chi_T^G(A,\vv,\ww) = \O(\Mo) \,.\right)
\end{equation}
\end{enumerate}
\end{Remark}
As a corollary of Theorem~\ref{thm:1}, we can take Hilbert-Poincaré series (character for the torus) of the equality in~\eqref{eq:thm1} and obtain a equality between numerical (power) series counting the graded dimensions. Formally, if $\Mo$ were compact, the Hilbert-Poincaré would be the pushforward to the point: $\smash{\ch_T: \KK_T(\Mo) \to \KK_T(\pt) = \Rring(T)}$, instead in general we land in the field of fractions (see, for example, \S4 in \cite{NaYo1})
\[
\ch_T : \KK_T(\Mo) \to \mathrm{Frac}(\Rring(T))=: \Qring(T)\,.
\]
\begin{Remark} If we consider the only fixed point for the torus action $0 \in \Mo$, and denote its inclusion by $\iota_0:\{0\} \to \Mo$, then by functoriality we have $\ch_T = (\iota_{0,*})^{-1}$, and this tells us that is not really necessary to invert all non-zero elements in $\Rring(T)$, but only the ones of the form $1-t^\beta$ for non-zero weights $\beta$, so that we actually land in the following smaller localisation (see \S2.1 and \S2.3 in \cite{Ok4}):
\[
\Rring(T)_{,\loc}:= \C \left[ t^\alpha, \frac{1}{1-t^\beta} \right]\, ,
\]
where $\alpha,\beta$ run over all weights of $T$ and $\beta\neq 0$.
\end{Remark}

Let us denote by $x \in T_\vv \subset G$ the variables in the maximal torus of the gauge group (\emph{Kähler variables}) and by $t=(a,\hbar) \in T=T_\ww \times T_\hbar$ the \emph{equivariant variables}. Then we have, by Weyl's integral formula:
\begin{equation}
\label{int1}
\ch_T\left( \chi_{T}^G (A,\vv,\ww) \right) = \frac{1}{|G|} \int\limits_{G} \ch_{G\times T} \left(\chi_T(A,\vv,\ww)\right)(g,t) d g = \frac{1}{|W|} \int\limits_{T_\vv} \ch_{T_\vv\times T} \left(\chi_T(A,\vv,\ww)\right))(x,t) \Delta(x) d x\, ,
\end{equation}
\begin{center}
($W$ is the Weyl group of $G$, $\Delta(x)$ is the Weyl factor,\\
and integrations are over the compact real forms of $G, T_\vv$)
\end{center}
Moreover, because the Euler characteristic of the homology of a complex is equal to the Euler characteristic of the complex itself, we have
\begin{equation}
\label{weights}
\ch_{T_\vv \times T} (\chi_T(A,\vv,\ww)) = \ch_{T_\vv \times T} \big( \O(M(Q,\vv,\ww) )\otimes \Lambda^\bullet \lieg \big) = \frac{\prod_i (1-\hbar_1\hbar_2 r_i) }{\prod_j (1-s_j)}\, .
\end{equation}
where $s_j$ are the weights of $M(Q,\vv,\ww)^*$ and $r_i$ are the weights of $\lieg$:
\[
\ch_{T_\vv\times T} (M(Q,\vv,\ww)) = \sum_j s_j^{-1}\, ,\qquad \ch_{T_\vv} (\lieg) = \sum_i r_i\, .
\]
To summarise:
\begin{Cor}
\label{cor:1}
Under the same conditions of Theorem~\ref{thm:1}, and with the notation used in the previous equations (in particular~\eqref{weights}), we have the following equality of Poincaré-Hilbert series in the field of fractions $\Qring(T)$:
\begin{equation}
\label{eqint1}
\ch_T \O(\M^\chi(Q,\vv,\ww)) = \ch_T \O (\Mo(Q,\vv,\ww)) = \frac{1}{|W|} \int_{T_\vv}  \frac{\prod_i (1-\hbar_1\hbar_2 r_i) }{\prod_j (1-s_j)} \Delta(x) d x \, .
\end{equation}
\end{Cor}

We calculate the above expression~\eqref{eqint1} in some concrete examples in \S~\ref{sec:examples}.
\begin{Remark}
\label{rem:1}
The right-hand side of~\eqref{eqint1} does not depend on the GIT parameter $\chi$, while the left-hand side a priori does. By picking different $\vv$-regular $\chi,\chi'$ we obtain a combinatorical identity 
\[
\ch_T \O(\M^\chi(Q,\vv,\ww)) = \ch_T\O(\M^{\chi'} (Q,\vv,\ww) )\, ,
\]
which we will show to be non-trivial, also in simplest quiver cases (see \S~\ref{sec:examples}, specifically Remark~\ref{rem:2} in \S~\ref{subsec:grass}).
\end{Remark}

\subsection{Other isotypical components and second integral formula}
\label{subsec:int2}
In this section we prove a result similar to Theorem~\ref{thm:1} to relate other tautological sheaves with the corresponding isotypical components.

Let us recall that to define $\M^\chi$ we fixed a character $\chi \in \Hom_{\Grp_\C}(G,\C^\times) $. This character defines a $1$-dimensional representation $\C_\chi$ of $G$, whose image under the Kirwan map is the Serre twisting sheaf
\begin{equation}
\kappa_T \left(\left[ \C_\chi \right] \right) = \big[ \O_{\M^\chi} (1) \big] \in K_T(\M^\chi)\, .
\end{equation}
For each $V_\lambda$ irreducible representation of $G$, we have a tautological sheaf $\mathcal{V}_\lambda$ in the $K$-theory of $\M^\chi$. By Serre vanishing theorem when we twist 
\begin{equation}
\mathcal{V}_\lambda(m) := \mathcal{V}_\lambda  \otimes  \O_{\M^\chi}(m) \, ,
\end{equation}
by a sufficiently large power $m \gg 0$ of the twisting sheaf, higher cohomology vanish, so that
\begin{equation}
\label{Serre}
\chi_T (\M^\chi, \mathcal{V}_\lambda(m) ) = \HH^0(\M^\chi,\mathcal{V}_\lambda(m))\, .
\end{equation}
Moreover, more or less by definition of the GIT quotient $\M^\chi$, this is equal to the $G$-invariant global sections of the trivial vector bundle $\smash{\underline{V_\lambda \otimes \C_{\chi^m} }}$ over the stable locus:
\begin{equation}
\HH^0(\M^\chi,\mathcal{V}_\lambda(m)) = \Gamma \left( \mu^{-1}(0)^{\chi\st}, \underline{V_\lambda \otimes \C_{\chi^m} } \right)^G \,.
\end{equation}
Finally for $m \gg 0$ large enough, the following natural restriction map becomes an isomorphism (see for example the proof of Lemma 3 in Appendix A of \cite{AgFrOk}):
\begin{equation}
\label{restri}
\Gamma \left( \mu^{-1}(0), \underline{V_\lambda \otimes \C_{\chi^m} } \right)^G  \xrightarrow{\sim} \Gamma \left( \mu^{-1}(0)^{\chi\st}, \underline{V_\lambda \otimes \C_{\chi^m} } \right)^G \, ,
\end{equation}
but the left-hand side is nothing else but 
\begin{equation}
\label{eqGla}
\Gamma \left(\mu^{-1}(0),\underline{V_\lambda \otimes \C_{\chi^m} } \right)^G = \left( \O(\mu^{-1}(0)) \otimes V_\lambda \otimes \C_{\chi^m} \right)^G \, .
\end{equation}
It is worth noticing at this point that irreducible representations $V_\lambda$ of $G$ are labelled by collections of partitions $\lambda = (\lambda^{(1)},\dots ,\lambda^{(n)})$ and that the representation $V_\lambda \otimes \C_{\chi^m}$ is still a irreducible representation of $G$, corresponding to the shifted collection of partitions:
\begin{equation}
\label{large}
V_\lambda \otimes \C_{\chi^m} = V_{\widetilde{\lambda}}\,, \qquad \widetilde{\lambda} := \lambda + m \underline{\theta} = (\lambda^{(1)} + m \underline{\theta_1} \,,\dots, \lambda^{(n)} + m\underline{\theta_n}) \quad (\chi =\chi_\theta)\, ,
\end{equation}
(see Appendix~\ref{app:B} for the notation). We give the following definition:
\begin{Defn}
\label{def:large}
We say that an irreducible representation $V_{\widetilde{\lambda}}$ is \emph{large enough} if $\widetilde{\lambda} = \lambda +m \underline{\theta}$ (see~\eqref{large}) with $m \gg 0$ large enough for both~\eqref{Serre} and~\eqref{restri} to be true. This notion depends on the quiver $Q$, on the dimension vectors $\vv,\ww$ and on the $\vv$-regular $\chi=\chi_\theta$.
\end{Defn}
Denoting by $\widetilde{\lambda}^*$ the partition corresponding to the dual representation, we can continue equation~\eqref{eqGla} to recognise:
\begin{equation}
 \left( \O(\mu^{-1}(0)) \otimes V_{\widetilde{\lambda}} \right)^G  =  \left( \O(\mu^{-1}(0)) \otimes V_{\widetilde{\lambda}^*}^* \right)^G = \HH_0(A,\vv,\ww)_{\widetilde{\lambda}^*}^G\, ,
\end{equation}
the isotypical component of $\widetilde{\lambda}^*$ of the (zeroth) representation homology. Finally if we observe that with flat moment map, higher homologies vanish, we obtain the following result:

\begin{Thm}
\label{thm:2}
Let $\vv,\ww$ be dimension vectors for which the moment map is flat, and fix $\chi=\chi_\theta$ with $\theta$ $\vv$-regular. For $\lambda$ large enough (in the sense of Definition~\ref{def:large}) we have 
\begin{equation}
\label{eq:thm2}
p_*( [\mathcal{V}_\lambda ] ) = [\HH^0(\M^\chi,\mathcal{V}_\lambda )] = \chi_T^{\lambda^*} (A,\vv,\ww) \in \KK_T(\M^0(Q,\vv,\ww))\, .
\end{equation}
\end{Thm}

The analogous integral formula to  obtained by taking characters is 
\begin{Cor}
\label{cor:2}
Under the same conditions of Theorem~\ref{thm:2}, and with the notation used in~\eqref{weights}, we have the following equality of Poincaré-Hilbert series in the field of fractions $\Qring(T)$:
\begin{equation}
\label{eqint2}
\ch_T(\chi_T(\M^\chi,\mathcal{V}_\lambda) )=
\ch_T ( \HH^0 (\M^\chi, \mathcal{V}_\lambda ) ) = \frac{1}{|W|} \int_{T_\vv}  \frac{\prod_i (1-\hbar_1\hbar_2 r_i) }{\prod_j (1-s_j)} f_\lambda(x) \Delta(x) d x\, .
\end{equation}
where $f_\lambda(x) = \ch_{T_\vv} (V_{\lambda})$ (it is the product of Schur polynomials associated to the partitions in $\lambda$).
\end{Cor}


\section{Examples}
\label{sec:examples}
In this section we explain some concrete examples, mainly from the easiest quivers already considered in the previous sections. We see how such elementary quivers still produce varieties of great interest in various fields of mathematics.

\subsection{Cotangent bundle of Grassmannian} 
\label{subsec:grass}
The quiver $Q=A_1$ with only one vertex and no arrows. The framed, doubled quiver has two vertices and two arrows connecting them in opposite directions.
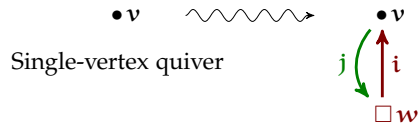
\begin{figure} [htbp]   
\centering
\scalebox{0.9}{
\begin{tikzpicture}
\node (vertex0){ $\,\,\,\,\bullet \, \bm v$};
\node [below= of vertex0, below=0.2](script) {Single-vertex quiver};
\node [right= of vertex0, right= 3cm](vertex){ $\,\,\,\,\bullet \, \bm v$};
\node [right= of vertex0, right= 0.3cm] (dummy0) {};
\node [left= of vertex, left= 0.3cm] (dummy) {};
\path (dummy0) edge[ ->, decorate, decoration=snake] (dummy);
\node [below= of vertex] (wertex) {{\color{myred} $\,\,\,\,\,\,\,\square\, \bm w$}};
\path (wertex) edge [very thick,myred, ->, out=90, in=270] node[right] {{\color{myred}$\bm i$}} (vertex);
\path (vertex) edge [very thick,mygreen, ->, out=230, in=130] node[left] {{\color{mygreen}$\bm j$}} (wertex);
\end{tikzpicture}
}
\caption{{\color{myred} Framing} and {\color{mygreen}doubling} the single-vertex quiver.}
\label{1vertex}
\end{figure}

Therefore:
\[
\mu^{-1}(0) = \{(I,J ) \in \Hom_\C(\C^w,\C^v) \oplus \Hom_\C(\C^v,\C^w) \,| \, I\circ J= 0 \}\, .
\]
Because we have only one vertex we have to choose the GIT parameter $\theta \in \Z$, and it is easy to check that the $v$-regularity condition means simply $\theta \neq 0$ (independently from $v$). For $\theta \neq 0$ we have the following identifications of the semistable locus:
\[
\theta\text{-semistable points} = 
\begin{cases}
J \,\,\,$\text{injective}$,\qquad & \theta <0\\
I \,\,\,$\text{surjective}$,\qquad & \theta >0
\end{cases}
\]
and the GIT quotient is isomorphic to the cotangent bundle $\TT^*\Gra(v,w)$ of $v$-planes in $\C^w$ in the case $\theta <0$ and to $\TT^*\Gra(w-v,w)$ in the case $\theta >0$. The two varieties are isomorphic to each other, but we have the following different identifications of the points in the Grassmannian:
\begin{center}
$\theta <0$: \quad $\im(J) \in \Gra(v,w)$\, ,
\end{center}
\begin{center}
$\theta >0$: \quad $\ker(I) \in \Gra(w-v,w)$\, .
\end{center}
The affine quotient can be identified (using some version of the fundamental theorem of invariant theory):
\[
\Mo = \Spec\big( \O(\mu^{-1}(0))^{\GL_v} \big) \cong \{ A \in M_{w \times w}(\C) \, | \, A^2= 0, \,\, \rk(A) \leq v \} \, ,
\]
where $A$ represents the composition $J\circ I :\C^w \to \C^w$. The condition on the rank is due to the fact that $A : W\to V \to W$ factorises through $V$, but sometimes it is superfluous. In fact in general $A^2=0$ forces already $\rk(A) \leq \floor*{w/2}$. The moment map is flat if and only if $2v-1 \leq w$ (see Examples~\ref{Ex4}), and only in this cases the projective morphism
\[
p: \TT^*\Gra(v,w) \to \Mo
\]
is a resolution of singularities.

\begin{figure} [htbp]   
\centering
\scalebox{0.7}{
\begin{tikzpicture}[line cap=round,line join=round,>=triangle 45,x=1cm,y=1cm]
\clip(-2.0667301858639084,-3.5108298775737685) rectangle (6.904177983559213,3.148437622799894);
\draw [rotate around={-0.016986952433784123:(-0.07048257416445844,-1.4495225109001144)},line width=1pt] (-0.07048257416445844,-1.4495225109001144) ellipse (1.1572896967611386cm and 0.07800969357369665cm);
\draw [rotate around={-0.016986952433750663:(4.353972850149957,1.8710917913669887)},line width=1pt] (4.353972850149957,1.8710917913669887) ellipse (1.1572896967609971cm and 0.07800969357372227cm);
\draw [rotate around={-0.016986952433751187:(4.412769599975073,-1.373535913594244)},line width=1pt] (4.412769599975073,-1.373535913594244) ellipse (1.1572896967610227cm and 0.078009693573676cm);
\draw [line width=1pt] (3.200855023207718,1.8780499110390612)-- (5.564542354006798,-1.3662678183465022);
\draw [line width=1pt] (5.511235542376604,1.8702178314060889)-- (3.2616552998405597,-1.365148089272453);
\draw [rotate around={-0.016986952433817152:(-0.14030371458138596,2.007867666517542)},line width=1pt] (-0.14030371458138596,2.007867666517542) ellipse (1.1572896967613584cm and 0.0780096935737114cm);
\draw [shift={(3.211722866909251,0.319254900318319)},line width=1pt]  plot[domain=2.4861135517361923:3.8330690954692126,variable=\t]({1*2.768553706340638*cos(\t r)+0*2.768553706340638*sin(\t r)},{0*2.768553706340638*cos(\t r)+1*2.768553706340638*sin(\t r)});
\draw [shift={(-3.9835293928979834,0.22047232382067172)},line width=1pt]  plot[domain=-0.5416307742601232:0.5847003029701225,variable=\t]({1*3.2231857145089258*cos(\t r)+0*3.2231857145089258*sin(\t r)},{0*3.2231857145089258*cos(\t r)+1*3.2231857145089258*sin(\t r)});
\draw [->,line width=1pt] (1.3743244348849317,0.3952414976242015) -- (2.8148448055919983,0.38764283789361326);
\draw (1.937309507816434,0.9172397489633574) node[anchor=north west] {$p$};
\draw [rotate around={0:(-0.15711976810674244,0.3078366783965148)},line width=1pt,color=yqqqqq] (-0.15711976810674244,0.3078366783965148) ellipse (0.6084018611377151cm and 0.10684862285025438cm);
\draw[line width=4pt] ;
\draw (-0.8487139768565976,-1.7719086716075592) node[anchor=north west] {$\mathcal{M}_\theta \cong \TT^*\mathbb{P}^1$};
\draw (2.502917447695727,-1.636802181078821) node[anchor=north west] {$\mathcal{M}_0 \cong \textrm{Spec} \left( \frac{\mathbb{C}[a,b,c] }{(a^2+bc)} \right)$};
\begin{scriptsize}
\draw [fill=yqqqqq] (4.389424142777053,0.2421593666230701) circle (1.5pt);
\end{scriptsize}
\end{tikzpicture}
}
\caption{The (real) picture of the case $(v,w)= (1,2)$: this is also known as Springer resolution of the nilpotent cone of $\slin_2(\C)$.}
\label{blowupT*p1}
\end{figure}
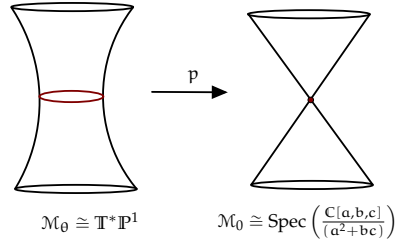

In this case in the torus $ T=T_w \times T_\hbar$ only the product $\hbar_1\hbar_2$ appears and we denote it by $\hbar$.
We can use~\eqref{eqint1} for $\chi=\chi_{-1}$ for which $\M^\chi=  \TT^*\Gra(v,w)$ and obtain a formula for the character of the ring of functions on the cotangent bundle of Grassmannian:
\begin{equation}
\label{int1}
\ch_T(\O(\TT^*\Gra(v,w))) =\frac{1}{v!} \cdot \oint_{|x_\alpha|=1} \frac{\prod_{\alpha,\beta} (1-\hbar x_\alpha^{-1}x_\beta )}{\prod_{\alpha,\gamma} (1-\hbar x_\alpha^{-1} a_\gamma ) (1-x_\alpha a_{\gamma}^{-1}) } \cdot \overbrace{\prod_{\alpha\neq \beta} (1-x_\alpha^{-1} x_\beta)}^{\Delta(x)} \, \overbrace{\prod_{\alpha} \frac{d x_\alpha }{2 \pi ix_\alpha}}^{d x}\, ,
\end{equation}
\begin{center}
where in the above $x=(x_\alpha) = (x_1,\dots ,x_v)$ and $a=(a_\gamma) =(a_1,\dots ,a_w)$.
\end{center}
The integral in the right-hand side can be computed by iterated residues, and by doing so we can recognise the localisation formula in equivariant $\KK$-theory as a sum over the fixed points $p \in (\TT^*\Gra(v,w))^T$ of the inverse of the $\KK$-theoretic Euler class of the tangent space at that point:
\begin{equation}
\label{fixed}
\ch_T(\O(\TT^*\Gra(v,w))) = \sum\limits_{\substack{B \subset \{1,\dots,w\} \\ \#B=v}} \frac{1}{\prod\limits_{\substack{\beta \in B \\ \gamma \notin B }} \left(1-\frac{a_\beta}{a_\gamma} \right) \left(1 -\hbar \frac{a_\gamma}{a_\beta} \right)} = \sum_{p \in  (\TT^*\Gra(v,w))^T} \frac{1}{\ch_T \big( \Lambda_{-1} \TT_{p}^* \left( \TT^*\Gra(v,w)\right) \big)}\, .
\end{equation}
For what concerns other sheaves, let us consider the standard representation $V=\C^v$ of $\smash{G=\GL_v(\C)}$. The associated tautological sheaf $\mathcal{V}$ on $\M^{\chi_{-1}}=\TT^*\Gra(v,w)$ is indeed the usual tautological sheaf of rank $v$. Irreducible representations are labelled by Schur functors $V_\lambda =\Schur_\lambda(V)$ where $\smash{\lambda= (\lambda_1 \geq \dots \geq \lambda_v)} $ is a integer partition of $v$ parts, and we consider the corresponding tautological sheaves $\mathcal{V}_\lambda$. For example the (standard) tautological sheaf itself is $\mathcal{V}=\mathcal{V}_{(1,0,\dots,0)}$, or powers of the Serre twisting sheaf are:
\begin{equation}
\label{ex}
\O_{\TT^*\Gra(v,w)}(m) =\mathrm{det}^{-m}(\mathcal{V}) = \mathcal{V}_{(-m,\dots,-m)}\, .
\end{equation}
A partition $\lambda$ becomes large (Definition~\ref{def:large}) in the sense that we can apply Theorem~\ref{thm:1} when all its components are negative enough (because the character $\chi=\chi_{-1}$ is negative), and it turns out that it suffices to have $\lambda_1 \leq 0$, that is equivalent to say that the partition is made of non-positive terms (an example is~\eqref{ex}, in which for $m>0$ the partition is negative and the corresponding sheaf has vanishing higher cohomologies). In this range we have 
\begin{equation}
\ch_T\HH^0\big(\TT^*\Gra(v,w),\mathcal{V}_\lambda  \big) =\frac{1}{v!} \cdot \oint_{|x|=1} \frac{ \big(\prod_{\alpha,\beta} (1-\hbar x_\alpha^{-1}x_\beta )\big) \,s_\lambda(x) }{\prod_{\alpha,\gamma} (1-\hbar x_\alpha^{-1} a_\gamma ) (1-x_\alpha a_{\gamma}^{-1}) } \cdot \prod_{\alpha\neq \beta} (1-x_\alpha^{-1} x_\beta) \, \prod_{\alpha} \frac{d x_\alpha }{2 \pi ix_\alpha}\, ,
\end{equation}
where $s_\lambda(x)=\ch_{T_v}(V_\lambda)$ is the Schur polynomial associated to the partition $\lambda$. Again, the integral in the right-hand side can be computed by means of iterated residues, giving the localisation formula for the corresponding tautological sheaf:
\begin{equation}
\ch_T\HH^0\big(\TT^*\Gra(v,w),\mathcal{V}_\lambda  \big) = \sum\limits_{\substack{B \subset \{1,\dots,w\} \\ \#B=v}} \frac{s_\lambda(a_B)}{\prod\limits_{\substack{\beta \in B \\ \gamma \notin B }} \left(1-\frac{a_\beta}{a_\gamma} \right) \left(1 -\hbar \frac{a_\gamma}{a_\beta} \right)} = \sum_{p \in  (\TT^*\Gra(v,w))^T} \frac{\ch_T(\mathcal{V}_\lambda)_{| p}}{\ch_T \big( \Lambda_{-1} \TT_{p}^* \left( \TT^*\Gra(v,w)\right) \big)}\, ,
\end{equation}
where the expression $s_\lambda(a_B)$ means that we are evaluating the Schur polynomial $s_\lambda(x_1,\dots,x_v)$ in the point $x=(a_\beta)_{\beta \in B}$. 

\begin{Remark}
\label{rem:2}
As already observed in Remark~\ref{rem:1} the right-hand side of the integral formula~\eqref{eqint1} does not depend on the character $\chi$, while a priori the left-hand side does. In~\eqref{int1} we used the character $\chi=\chi_{-1}$ for which $\smash{\M^\chi =\TT^*\Gra(v,w)}$. If we use $\chi'=\chi_{1}$ we have $\smash{\M^{\chi'} = \TT^*\Gra(w-v,w)}$. The fixed point formula for the first variety~\eqref{fixed} can be compared with the one for the second variety, and it gives a non-trivial combinatorical identity:
\begin{equation}
 \sum\limits_{\substack{B \subset \{1,\dots,w\} \\ \#B=v}} \frac{1}{\prod\limits_{\substack{\beta \in B \\ \gamma \notin B }} \left(1-\frac{a_\beta}{a_\gamma} \right) \left(1 -\hbar \frac{a_\gamma}{a_\beta} \right)} =  \sum\limits_{\substack{B \subset \{1,\dots,w\} \\ \#B=v}} \frac{1}{\prod\limits_{\substack{\beta \in B \\ \gamma \notin B }} \left(1-\frac{a_\gamma}{a_\beta} \right) \left(1 -\hbar \frac{a_\beta}{a_\gamma} \right)} \, .
\end{equation}
\end{Remark}


\subsection{Framed moduli space of torsion free sheaves on $\PP^2$} This is the case of the Jordan quiver, the quiver with one vertex and one loop, Figure~\ref{example}. Therefore:
\[
\mu^{-1}(0) = \{(X,Y,I,J ) \in \Hom_\C(\C^v,\C^v)^{\oplus 2} \oplus \Hom_\C(\C^w,\C^v) \oplus \Hom_\C(\C^v,\C^w) \,| \, [X,Y] + I\circ J = 0 \}\, .
\]
For GIT paramater $\theta \in \Z$:
\[
\theta\text{-semistable points} = 
\begin{cases}
\nexists 0 \neq S \subset V  \,\, \,\text{s.t.} \,\,\, \C\langle X,Y \rangle (S) \subset S \,\,\, \text{and}\,\,\, S \subset \ker(J), \qquad & \theta <0\\
\nexists S\subsetneq V  \,\, \,\text{s.t.} \,\,\, \C\langle X,Y \rangle (S) \subset S \,\,\, \text{and}\,\,\, \im(I) \subset S, \qquad & \theta >0
\end{cases}
\]
In both cases we have an identification between the Nakajima variety $\M^\theta$ and $M(w,v)$, the (framed) moduli space of torsion free sheaves on $\C\PP^2$ of rank $w$, second Chern class $c_2=v$, and fixed trivialisation at the line at $\infty$. The affine Nakajima variety is $\Mo \cong M_0(w,v)$ the framed moduli space of ideal instantons on $S^4 = \C^2 \cup \{\infty\}$. The map $\smash{p:M(w,v) \to M_0(w,v)}$ is always a resolution of singularities because the moment map is always flat.

When the framing is $w=1$ we obtain the Hilbert-Chow morphism from the Hilbert scheme of $v$ points on $\C^2$ to the symmetric $v$-power:
\[
p: \mathrm{Hilb}_v(\C^2) \to \Sym^v (\C^2)\, .
\]
For general $v$ and $w$ the integral formula looks like:
\begin{equation}
\label{int2}
\ch_T\O( M(w,v)) =\frac{1}{v! }\cdot \oint_{|x|=1} I (x,a,\hbar)  \cdot \prod_{\alpha\neq \beta} (1-x_\alpha^{-1} x_\beta) \, \prod_{\alpha} \frac{d x_\alpha }{2\pi i x_\alpha}\, ,
\end{equation}
where 
\[ 
I (x,t,\hbar)= 
 \frac{\prod_{\alpha,\beta} (1-\hbar_1 \hbar_2 x_\alpha^{-1}x_\beta )}{ \prod_{\alpha,\beta} (1-\hbar_1x_\alpha^{-1} x_\beta )(1-\hbar_2 x_\alpha^{-1} x_\beta ) \cdot \prod_{\alpha,\gamma} (1-\hbar_1\hbar_2 x_\alpha^{-1} t_\gamma ) (1-x_\alpha t_{\gamma}^{-1}) } \, ,
 \]
and it is also known as the integral formula for Nekrasov partition function (proved for example in Appendix A of \cite{FeMu}).
 
For other isotypical components, let us say that we fixed $\chi=\chi_1$. Again we have a tautological sheaf of rank $v$, $\mathcal{V}$, and other sheaves associated to irreducible representations are labelled by Schur functors $\mathcal{V}_\lambda$ where $\lambda$ is an integer partition of $v$ parts. In this case the largeness condition indeed means that the partition is big enough, and it turns out that it suffices for it to be non-negative $\lambda_1\geq \dots \geq \lambda_v \geq 0$. In this range we have:
\begin{equation}
\label{this}
\ch_T \HH^0(M(w,v), \mathcal{V}_\lambda) =\frac{1}{v!} \oint_{|x|=1} I (x,a,\hbar)  \cdot  s_\lambda(x)\cdot \prod_{\alpha\neq \beta} (1-x_\alpha^{-1} x_\beta) \, \prod_{\alpha} \frac{d x_\alpha }{2\pi i x_\alpha}\, .
\end{equation}
For $\lambda \geq 0$ the Schur polynomial $s_\lambda(x)$ is indeed an actual polynomial (and not a Laurent polynomial), and therefore with~\eqref{this} we recover the integral formula for Nekrasov partition function with matter fields (the matter field is represented by the sheaf $\mathcal{V}_\lambda$ in this case) which was proved for example in \cite{Mu}.
\subsection{Symplectic dual of $\TT^*\PP^{n-1}$}

$X=\TT^*\Gra(k,n)$ has a symplectic dual, $X\dual$, which for the choice of parameters $2k \leq n$ can be shown to be also a Nakajima quiver variety (\cite{RiSmVaZh}). Specifically it is the Nakajima variety associated to the following $A_{n-1}$ quiver, with dimension vectors:
\[
\begin{cases}
\vv= (1,2, \dots, k-1, \underbrace{k,\dots, k}_{(n-2k+1)\text{-times}}, k-1, \dots, 2,1)\, ,\\
\ww= (w_1,\dots,w_{n-1}) \qquad w_i = \delta_{i,k} + \delta_{i,n-k}\, .
\end{cases}
\]
We restrict to the case $k=1$, for which dimension vectors are 
\begin{equation}
\label{dimension}
\begin{cases}
\vv= (1, \dots , 1)\, ,\\
\ww=(1,0,\dots,0,1)\, ,
\end{cases}
\end{equation}
and the corresponding Nakajima quiver variety is the symplectic dual of $\TT^*\PP^{n-1}$. For $n=2$ we go back to the $A_1$ case with dimensions $v=1$ and $w=2$, so we find that $\TT^*\PP^1$ is symplectic dual to itself. Let us study the other cases $n \geq 3$ which are different.

As usual we denote the arrows in the quiver by $x_1,\dots,x_{n-2}$, their dual by $y_1,\dots ,y_{n-2}$ and then we have $i_1,j_1$ and $i_{n-1},j_{n-1}$ because of the non-trivial framing at the vertices $1$ and $n-1$. The zero locus of the moment map is the following algebraic variety in a $2n$-dimensional affine space
\[
\mu^{-1}(0) \cong \Spec\left(\frac{ \C [x_1,y_1, \dots ,x_{n-2},y_{n-2},i_1,j_1,i_{n-1},j_{n-1} ] }{i_1j_1 = x_1 y_1= x_2 y_2= \dots =x_{n-2}y_{n-2} = -i_{n-1}j_{n-1}}\right) \, .
\]
The gauge group is a $n-1$-dimensional torus $G_\vv= \GL_1(\C)^{n-1} = (\C^\times)^{n-1}$, and the affine Nakajima variety is identified with the ADE singularity of type $A_{n-1}$:
\begin{equation}
\label{duval}
\Mo \cong \Spec \left( \frac{\C[x,y,z]}{xy=z^{n}}\right) \cong \C^2 \slash{\Z_n}\, ,
\end{equation}
where $x=x_1 \cdots x_{n-2} i_1 j_{n-1}$, $y=y_1\cdots y_{n-2} i_{n-1}j_1$, $z=x_1 y_1$. We recall that the action $\smash{\Z_n\acts \C^2}$ that gives the corresponding ADE singularity of type $A_{n-1}$ is given by the embedding $\smash{\Z_n \subset \SL_2(\C)}$ in which a $n$-th root of unity $\xi\in \Z_n$ becomes the matrix $\mathrm{diag}(\xi,\xi^{-1})\in \SL_2(\C)$.

We fix GIT parameter $\chi = \chi_{\theta_+}$ with $\theta_+= (1,1,\dots,1)$. The corresponding smooth Nakajima quiver variety is a consecutive ($n-1$ times) blowup of the singular point $x=y=z=0$ in~\eqref{duval}:
\begin{equation}
\label{nakajimaduval}
p: \M^{\chi_+} =\widetilde{\C^2\slash \Z_n} \xrightarrow{\qquad} \Mo = \C^2\slash\Z_n\, ,
\end{equation}
with exceptional fiber $p^{-1}(0)$ given by $n-1$ copies of Riemann spheres $\PP^1$ intersecting in such a way that their underlying intersection graph is $A_{n-1}$ (see \cite{duV}), as shown in Figure~\ref{spheres}.
\begin{figure}[H]
\centering
\includegraphics[scale=0.13]{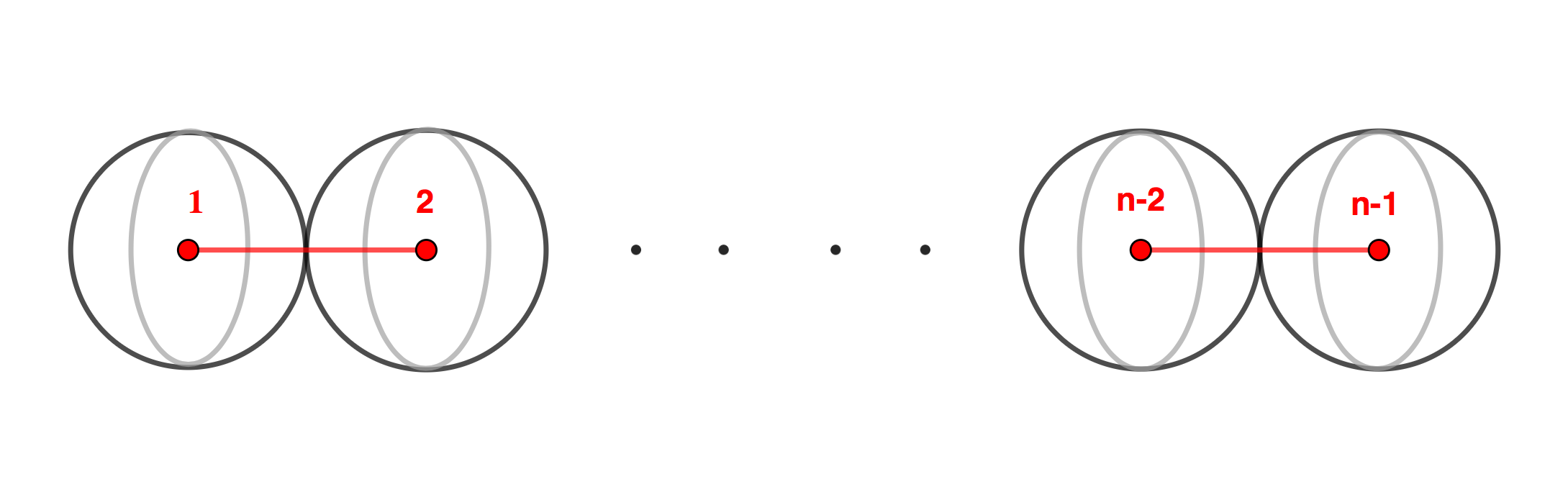}
\caption{Every sphere is replaced by a vertex and two vertices are linked by as many arrows as intersection points of the corresponding spheres.}
\label{spheres}
\end{figure}

The associated derived representation scheme is
\[
\DRep_{\vv,\ww}=\Spec \left( \C [x_1,y_1,x_2,y_2,\dots, x_{n-2},y_{n-2}, i_1,j_1,i_{n-1},j_{n-1}, \vartheta_1,\dots, \vartheta_{n-1} ] \right)\, ,
\]
where $\vartheta_i$ have homological degree $1$ and differential
\begin{equation}
\label{diff}
\begin{cases}
d \vartheta_1 = -y_1x_1 + i_1j_1\, ,\\
d \vartheta_k = x_{k-1}y_{k-1} - y_k x_k\,, \quad (k=2,\dots, n-2)\, ,\\
d \vartheta_{n-1} = x_{n-2}y_{n-2} + i_{n-1}j_{n-1}\, ,
\end{cases}
\end{equation}
and they are invariants under the gauge group $G_\vv=\GL_1^{n-1}$, so that the associated character scheme is simply
\[
\DRep_{\vv,\ww}^{G_\vv} \cong \Spec \left( \left(\frac{\C[x,y,z_1,\dots,z_{n-2},z_{n-1},z_n] }{xy=z_1 \cdots z_n }\right) [\vartheta_1,\dots,\vartheta_{n-1}]\right)\, ,
\]
where $x,y$ are the same classes as before in~\eqref{duval}, $z_k = x_k y_k$ for $k=1,\dots,n-2$, $z_{n-1} =i_1j_1$, $z_n=i_{n-1}j_{n-1}$. We denote the variables in the equivariant torus $T = T_{\ww}\times T_\hbar$ by $(a,\tilde{a},\hbar_1,\hbar_2)$ (where $a$ is on the vertex $1$ and $\tilde{a}$ on the vertex $n-1$) and we have:
\begin{equation}
\ch_T \O \left( \widetilde{\C^2\slash \Z_n} \right) = \ch_T \O \left( \C^2\slash\Z_n \right) = \ch_T\left( \chi_T(\DRep_{\vv,\ww}^{G_\vv}) \right)  = \frac{1+\hbar_1 \hbar_2\dots+ \hbar_1^{n-1}\hbar_2^{n-1}}{\left( 1- \hbar_1^{n-1} \hbar_2 \frac{a}{\tilde{a}}\right) \left(1- \hbar_1 \hbar_2^{n-1} \frac{\tilde{a}}{a}\right)}\, .
\end{equation}

\appendix 

\section{Projective model structure on T-equivariant dg-algebras}
\label{app:A}

In this Appendix we give a proof of Theorem~\ref{thm:T-homtheory} that gives a projective-like model structure on the category of $T$-equivariant dg-algebras $(\DGA_k^+)^T$, for an algebraic torus $T=(k^\times)^r$. We use the same strategy used in~\cite{BeRa}, in which the authors prove that the category of \emph{bigraded} dg-algebras $\mathtt{BiDGA}_k$ has a projective-like model structure\footnote{Which ultimately follows the explicit proofs of the existence of the projective model structure on $\DGA_k^+$ by \cite{Mun} or \cite{FeHaTh}.}. The key observation is to recognise that $\mathtt{BiDGA}_k$ being the category of dg-algebras with an additional non-negative (polynomial) compatible grading, is equivalent to the category of $T$-equivariant dg-algebras with a polynomial torus action (i.e. weight spaces are only for non-negative weights), and that the polynomial condition can be dropped, and substituted by the rational condition, in which weights can be arbitrary integers.

More precisely, weight spaces for a torus $T=(k^\times)^r$ are $r$-tuples of integers $n \in \Z^r$, and we observe that the category of dg-algebras with a rational $T$-action $(\DGA_k^+)^T$ is equivalent to the category of dg-algebras $A \in \DGA_k^+$ with:
\begin{enumerate}
\item An additional grading of the underlying chain complex
 $\smash{A=\oplus_{n\in \Z^r}} A(n)$. This means that each $A(n)$ is a complex of vector spaces preserved by the differential in $A$: $d A(n) \subset A(n)$. 
\item The grading is compatible with the multiplication in $A$: $A(n) \cdot A(m) \subset A (n+m)$.
\end{enumerate}
In fact, on one hand if $A \in (\DGA_k^+)^T$ then for $n \in \Z^r$ we define $\smash{A(n)=\{ a \in A \,| \, t\cdot a = t^n a, \,\, \forall t \in T\} }$ as the corresponding weight space and the above 2 conditions are satisfied thanks to the rationality of the action (recall, Definition~\ref{def:rationalaction}). On the other hand, obviously if we have such a decomposition we define the $T$-action on $A$ accordingly by $t\cdot a := \sum_{n} t^n a(n)$, where $a=\sum_n a(n)$, and the resulting $T$-action is rational.

The observation that $(\DGA_k^+)^T$ is equivalent to the category of dg-algebras with an additional grading as described above will be also useful later, and we will use indifferently one or the other property, according to what is more convenient from time to time.

Let us also denote by $k[T]=\O(T)$, a Laurent polynomial ring in $r$ variables and observe that 
\begin{Lem}
\label{lem:adjkt}
The forgetful functor $U: (\DGA_k^+)^T \to \DGA_k^+$ is left-adjoint to the ``free $T$-equivariant extension'' functor: 
\begin{equation}
\adjunct{(\DGA_k^+)^T}{\DGA_k^+}{U}{k[T]\otimes (-)}\, .
\end{equation}
\end{Lem}
\begin{proof}
The adjunction is given by the natural isomorphisms:
\[
\Hom_{\DGA_k^+} (UA,B) \cong \Hom_{(\DGA_k^+)^T} (A,k[T]\otimes B)\, ,
\]
where to a $T$-equivariant morphism $\varphi: A \to k[T]\otimes B$ we assign the composition with the evaluation map at $1 \in T$:
\[
A \xrightarrow{\varphi} k[T]\otimes B \xrightarrow{\ev_1 \otimes 1_B} k\otimes B \cong B\, .
\]
Conversely if we start from a map $f: UA \to B$ which is not necessarily $T$-equivariant, we can construct a $T$-equivariant map $\varphi: A \to k[T]\otimes B$ by decomposing:
\[
\varphi : \bigoplus\limits_{n\in \Z^r} A(n) \to \bigoplus\limits_{n \in \Z^r} B \cdot t^n\, ,
\]
and defining $\varphi_{|_{A(n)}} : A(n) \to B\cdot t^n $ as $f_{|_{A(n)}}(-) \cdot t^n$.
\end{proof}
In order to prove Theorem~\ref{thm:T-homtheory} we need a few definitions and lemmas. Throughout this section of the Appendix we denote by $\CC=\DGA_k^+$ and by $\CC^T = (\DGA_k^+)^T$.
\begin{Notation}
We denote by $\COF,\WE,\FIB$ the collection of cofibrations, weak equivalences, and fibrations in the projective model structure on $\CC$. So $\FIB$ are surjective maps in positive homological degrees, $\WE$ are the quasi-isomorphisms, and $\COF=\llp(\WE \cap \FIB) $, where $\llp(-)$ denotes the collection of morphisms with the left lifting property with respect to another collection of morphisms. Finally recall that a fibration which is also a quasi-isomorphism is actually surjective in \emph{all} homological degrees, so that $\WE \cap \FIB$ consists of surjective quasi-isomorphisms.
\end{Notation}
\begin{Defn}
\label{def:tate}
A morphism $i: S \to R \in \CC^T$ is a \emph{$T$-equivariant noncommutative Tate extension} (also simply a Tate extension) if there is a (possibly infinite) sequence $V^{(0)} \subset V^{(1)} \subset V^{(2)} \subset \dots$ of (homologically) graded, $T$-equivariant vector spaces such that
\begin{enumerate}
\item Each $S \ast_k T(V^{(i)})$ has a differential and a compatible embedding $S \ast_k T(V^{(i)}) \subset R$ such that at the limit $V=\cup_i V^{(i)}$:
\[
S \ast_k TV = \mathop{\lim_{{\longrightarrow}_i}}  S \ast_k T(V^{(i)}) = R\, .
\]
\item Each differential has the property that $d (V^{(i)}) \subset S \ast_k T(V^{(i-1)})$ (and for $i=0$, $d (V^{(0)}) \subset S$).
\end{enumerate}
We denote the collecion of such morphism by $\TE \subset \Mor(\CC^T)$.
\end{Defn}

\begin{Lem}
\label{lem:tate}
\begin{enumerate}[label=(\roman*)]
\item Every morphism $S \to A$ in $\CC^T$ has a factorisation of the form $S\xrightarrow{i} R \xrightarrow{p} A $ where $i \in \TE$ and $p\in U^{-1}(\WE\cap \FIBB )$ (is a surjective quasi-isomorphism).
\item Every Tate extension has the left lifting property with respect to morphisms that are surjective quasi-isomorphisms: $\TE \subset \llp(U^{-1}(\WE \cap \FIBB))$.
\end{enumerate}
\end{Lem}

For the proof one can check that the proof of Proposition 3.1 (which relies on Proposition 2.1(ii)) of \cite{FeHaTh} can be used also in this case of $T$-equivariant (i.e. additionally graded) objects.

Now let $x$ be a variable of positive homological degree as well as of some weight $n \in \Z^r$ for the torus $T$, and set $V_x := [ 0 \to k\cdot x \to k\cdot dx \to 0 ]$, and its tensor algebra $T(V_x) \in \CC^T$. Extensions by objects of this form play another important role:

\begin{Defn}
\label{def:special}
A morphism in $\CC^T$ of the form $\smash{S \to S \ast_k \coprod_{i \in I } T(V_{x_i})}$, where $I$ is any, possibly uncountably infinite, indexing set, is called a \emph{special extension}. We denote the collection of special extensions by $\SE \subset \Mor(\CC^T)$.
\end{Defn}

\begin{Lem}
\label{lem:special}
\begin{enumerate}[label=(\roman*)]
\item Every morphism $S \to A$ in $\CC^T$ has a factorisation of the form $S\xrightarrow{i} R \xrightarrow{p} A$ where $i \in \SE$ and $p \in U^{-1}( \FIBB)$.
\item $\SE \subset \llp ( U^{-1} (\FIBB))$.
\item $\SE \subset U^{-1}(\WE)$.
\end{enumerate}
\end{Lem}

\begin{proof} 
(i) It suffices to consider the set of elements of $A$ of positive homological degrees as well as of some weight for the torus action: $I:= \{ a \in A(n)_i \, | \, n \in \Z^r, \, i  > 0\}$. For each $a \in I$ we consider the obvious $\smash{TV_{x_a} \xrightarrow{p_a} A}$ given by $p_a(x_a) =a$ (and consequently $p_a(dx_a) = da$). Then
\[
S \xrightarrow{i} S \ast_k  \coprod_{a \in I } T(V_{x_a}) \xrightarrow{f\ast_k \coprod\limits_{a \in I} p_a} A
\]
is the desired factorisation. (ii) and (iii) are quite obvious.
\end{proof}

Now we have everything we need to prove that the following definition yields the desired model structure on $\CC^T$:
\begin{Defn}
\label{def:Tmodel}
We define weak equivalences, fibrations and cofibrations in $\CC^T$ as:
\begin{equation}
\WE^T:=U^{-1}(\WE) \,,\qquad \FIB^T:= U^{-1}(\FIB) \,,\qquad  \COF^T:= \llp (\WE^T \cap \FIB^T)=\llp( U^{-1}( \WE \cap \FIB))\,.
\end{equation}
\end{Defn}

We observe that, by Lemma~\ref{lem:tate}(ii), Tate extensions are cofibrations: $\smash{\TE \subset \COF^T}$, and by Lemma~\ref{lem:special}, special extensions are acyclic cofibrations: $\smash{\SE \subset \WE^T\cap \COF^T}$. In fact, it is useful to observe that 
\begin{Prop}
\label{prop:retract}
Every acyclic cofibration in $\CC^T$ is a retract of a special extension.
\end{Prop}
\begin{proof}
Let $i:A \to B$ be an acyclic cofibration and let us factor it as $\smash{A \xrightarrow{\widetilde{i}} R \xrightarrow{q} B}$ where $\widetilde{i}$ is a special extension and $q$ is a fibration, according to Lemma~\ref{lem:special}(i). $q$ is also a weak equivalence, because of the 2-out-of-3 property (see (MC2) in the proof of the next Theorem), therefore it is an acyclic fibration, and we can find a lift of the diagram:
\begin{equation}
\begin{tikzcd} 
  A  \arrow[r, "\widetilde{i}"] \arrow[d, "i"]
    & R \arrow[d, "q"]  \\
 B \arrow[r, "1_B"] \arrow[ru, dashrightarrow,"\exists l"]
 &B
\end{tikzcd}
\end{equation}
which proves that $i$ is a retract of the special extension $\widetilde{i}$:
\begin{equation}
\begin{tikzcd} 
  A  \arrow[r, "1_A"] \arrow[d, "i"]  & A \arrow[r, "1_A"] \arrow[d, "\widetilde{i}"] &A \arrow[d,"i"] \\
 B \arrow[r, "l"] \arrow[rr, bend right=35, "1_B"] & R\arrow[r, "q"] & B
\end{tikzcd}
\end{equation}
\end{proof}

\begin{Thm*}[\ref{thm:T-homtheory}]
\begin{enumerate}
\item Definition~\ref{def:Tmodel} defines a model structure on $\CC^T$.
\item The forgetful functor $U :\CC^T \to \CC$ preserves cofibrations.
\end{enumerate}
\end{Thm*}
\begin{proof}
(1) (MC1) (notation of Definition 3.3 of \cite{DwSp}): finite limits and colimits exist in $\CC^T$ because equalizers, coequalizers, finite product and finite coproducts exist (the same constructions as in $\CC$ work in the equivariant setting). (MC2) $\WE^T$ has the 2-out-of-3 property because it is $U^{-1}(\WE)$ with $\WE$ having the 2-out-of-3 property. (MC3) $\WE^T$ and $\FIB^T$ are closed under retracts because, again, defined as $U^{-1}$ of classes closed under retracts. $\COF^T $ are closed under retracts because they are defined as the morphisms with the left lifting property with respect to some class $\llp {\mathcal{A}}$ , and this is always closed under retracts (it does not matter what $\mathcal{A}$ is). (MC4) We need to prove that for a diagram in $\CC^T$ of the following form:
\begin{equation} 
\begin{tikzcd}
  A  \arrow{r}{f} \arrow{d}{i}
    & C \arrow{d}{p}  \\
 B \arrow{r}{g}
&D
 \end{tikzcd}
\end{equation}
a lift exists in the following situations: (i) $i$ is a cofibration and $p$ is an acyclic fibration, (ii) $i$ is an acyclic cofibration and $p$ is a fibration. (i) is obviously true by the definition of cofibrations. To prove that a lift exists in the case (ii), thanks to Proposition~\ref{prop:retract}  we only need to find a lift when $i$ is a special extension, but this is true by Lemma~\ref{lem:special}(ii). (MC5) We need to prove that each morphism $S \to A$ in $\CC^T$ has factorisations of the form: (i) cofibration followed by an acyclic fibration, (ii) acyclic cofibration followed by a fibration. (i) follows from Lemma~\ref{lem:tate}(i), and (ii) follows from Lemma~\ref{lem:special}(i).

(2) This follows from the fact that $U$ is left adjoint to $k[T]\otimes (-)$ (Lemma~\ref{lem:adjkt}), and the latter preserves weak equivalences and fibrations, therefore $U$ preserves cofibrations.
\end{proof}

\section{Representation theory of $G=G_\vv$}
\label{app:B}

In this Appendix we recall the theory of irreducible representations of (a product of) general linear groups and we fix the notation. Polynomial irreducible representations of $\GL_v(\C)$ are labelled by ordinary (non-negative) partitions $\smash{\lambda=(\lambda_1,\dots, \lambda_v)}$. More precisely, they are obtained by applying the Schur functors $\Schur_\lambda(-): \Vect_\C \to \Vect_\C$ to the standard representation $V=\C^v$:
\begin{equation}
\label{Vla}
\Schur_\lambda(V)\, .
\end{equation}
Their characters, the Schur polynomials, form a linear basis of the ring of symmetric polynomials in $v$ variables:
\begin{equation}
s_\lambda (x) := \ch( \Schur_\lambda(V)) \in \Z [ x_1,\dots ,x_v]^{\Sigma_v}\, .
\end{equation}
Examples are 
\begin{enumerate}
\item $V=\C^v$ itself is $\Schur_{(1,0,\dots,0)}(V)$ and $s_{(1,0,\dots,0)}(x) = x_1 + \dots + x_v$.
\item More generally $\Schur_{(d,0,\dots,0)}(V) = \Sym^d (V)$ and $s_{(d,0,\dots,0)}(x) =h_d(x)$ the complete symmetric polynomial.
\item For $\lambda=(1,1, \dots, 1,0,\dots 0)$ with $1$ repeated $d$-times, $\Schur_\lambda(V) = \Lambda^d(V)$ and $s_\lambda(x) = e_d(x)$ the elementary symmetric polynomial.
\item $1$-dimensional representations are given by $\underline{m}:= (m,m,\dots ,m)$, for which $\smash{\Schur_{\underline{m}}(V) = \det(V)^{m }}$, and $s_\lambda(x) = e_v(x)^m = x_1^m \cdots x_v^m$.
\end{enumerate}
If we shift a partition $\lambda$ to $\lambda+ \underline{m}:=(\lambda_1 +m,\dots, \lambda_v +m)$ we have
\begin{equation}
\Schur_{\lambda + \underline{m}}(V) = \Schur_\lambda(V) \otimes \det(V)^{ m }\, ,
\end{equation}
which allows to extend the definition of Schur functors to partitions made possibly of some negative parts $\smash{ \lambda \in  \Part_v:= \{ \lambda \in \Z^v \, | \, \lambda_1 \geq \dots \geq \lambda_v\} } $ as
\begin{equation}
\label{rala}
\Schur_\lambda(V) := \Schur_{\lambda - \underline{\lambda_v}}(V) \otimes \det(V)^{\lambda_v}\, .
\end{equation}
All irreducible rational representations of $\GL_v(\C)$ are of the form~\eqref{rala} for some integer-valued partition $\lambda \in  \Part_v$. Their characters are generalised Schur polynomials and they form a linear basis of the ring of symmetric Laurent polynomials in $v$ variables:
\begin{equation}
s_\lambda(x) = \ch (\Schur_\lambda(V)) \in \Z[x_1,x_1^{-1},\dots, x_v,x_v^{-1} ]^{\Sigma_v}\, .
\end{equation}
If now $\vv=(v_1,\dots, v_n)$ is a dimension vector and $\smash{G_\vv =\prod_i \GL_{v_i}(\C) }$ is a product of general linear groups, then its irreducible rational representations are labelled by $n$-tuples of partitions $\smash{\lambda = (\lambda^{(1)}, \dots ,\lambda^{(n)})\in \prod_i \Part_{v_i}}$, as the external tensor product of Schur modules:
\begin{equation}
V_\lambda:= \Schur_{\lambda^{(1)}}(\C^{v_1}) \boxtimes \dots \boxtimes \Schur_{\lambda^{(n)}} (\C^{v_n}) \, .
\end{equation}
Their characters are products of (generalised) Schur polynomials and we denote them by (the same notation as in~\eqref{eqint2}):
\begin{equation}
f_\lambda(x) : = \ch( V_\lambda) = s_{\lambda^{(1)}} (x^{(1)}) \cdots s_{\lambda^{(n)}} (x^{(n)}) \, ,
\end{equation}
where $x=(x^{(1)},\dots ,x^{(n)}) $ and each $x^{(i)}$ is a set of $v_i$ variables: $\smash{x^{(i)}= (x^{(i)}_1,\dots ,x^{(i)}_{v_i})}$.



\bibliographystyle{amsplain}
\bibliography{Derived}

\end{document}